\documentclass[12pt]{amsart}

\usepackage{array}

\usepackage{booktabs}

\usepackage{tabularx}
\usepackage{arydshln}
\usepackage{multirow} 
\usepackage{enumitem}
\usepackage{latexsym}
\usepackage{cite}
\usepackage[centertags]{amsmath}
\usepackage{amsfonts}
\usepackage{dsfont}
\usepackage{amssymb}
\usepackage{amsthm}
\usepackage{newlfont}
\usepackage{graphics}
\usepackage{color}
\usepackage{float}
\usepackage{diagbox}
\usepackage{longtable}
\usepackage{rotating}
\usepackage{multirow}
\usepackage{extarrows}
\usepackage[sort,compress,numbers]{natbib}
\usepackage[utf8]{inputenc}
\usepackage{collectbox}
\usepackage{hyperref}

\numberwithin{equation}{section}
\newtheorem{theo}{Theorem}
\newtheorem{defn}[theo]{Definition}
\newtheorem{exam}[theo]{Example}
\newtheorem{lem} [theo]{Lemma}
\newtheorem{cor}[theo]{Corollary}
\newtheorem{prop}[theo]{Proposition}
\newtheorem{rem}[theo]{Remark}
\newtheorem{algor}[theo]{Algorithm}
\newtheorem{remark} [theo]{Remark}
\newtheorem{prob}[theo]{Problem}
\numberwithin{equation}{section}

\makeatletter \@addtoreset{equation}{section}

\def\CT{\mathop{\mathrm{CT}}}
\def\x{\mathbf{x}}
\def\b{\mathbf{b}}
\def\N{\mathbb{N}}
\def\Q{\mathbb{Q}}
\def\Z{\mathbb{Z}}
\def\C{\mathbb{C}}
\def\R{\mathbb{R}}

\def\y{\mathbf{y}}
\def\0{\mathbf{0}}
\def\rank{\mathrm{rank}}

\def\sgn{\mathrm{sgn}}
\def\c{\mathfrak{c}^+}
\def\d{\mathfrak{c}^-}

\def\ord{\mathrm{ord}}
\def\1{\mathbf{1}}
\def\cs{{\mathbf{s}}}
\def\mydot{\raisebox{1mm}{$\centerdot$}}

\def\ind{\mathrm{ind}}
\title{A combinatorial simplicial cone decomposition}

\author{guoce xin, xinyu xu, zihao zhang}

\address{ $^{1,2,3}$School of Mathematical Sciences, Capital Normal University,
 Beijing 100048, PR China}
\email{$^1$\texttt{guoce.xin@163.com }\ \& $^2$\texttt{xinyu0510x@163.com}\& $^2$\texttt{zihao-zhang@foxmail.com}}
\date{3.24.2022}
\thanks{$*$ This work was partially supported by NSFC(12071311).}


\begin{document}

\begin{abstract}

This paper introduces an algebraic combinatorial approach to simplicial cone decompositions, a key step in solving inhomogeneous linear Diophantine systems and counting lattice points in polytopes. We use constant term manipulation on the system \( A\alpha = \mathbf{b} \), where \( A \) is an \( r \times n \) integral matrix and \( \mathbf{b} \) is an integral vector. We establish a relationship between special constant terms and shifted simplicial cones. This leads to the \texttt{SimpCone[S]} algorithm, which efficiently decomposes polyhedra into simplicial cones. Unlike traditional geometric triangulation methods, this algorithm is versatile for many choices of the strategy \( \texttt{S} \) and can also be applied to parametric polyhedra.
The algorithm is useful for efficient volume computation of polytopes and can be applied to address various new research projects.
Additionally, we apply our framework to unimodular cone decompositions. This extends the effectiveness of the newly developed \texttt{DecDenu} algorithm from denumerant cones to general simplicial cones.
\end{abstract}
\maketitle

\noindent
\begin{small}
\end{small}

\noindent
\begin{small}
\end{small}

\section{Introduction}

%

Throughout this paper, $A$ is always an $r\times n$ matrix with integer entries in $\Z$, and
$\b$ is an integral vector in $\Z^r$, unless specified otherwise.

A fundamental problem in combinatorics is to solve the inhomogeneous linear Diophantine system $A\alpha = \mathbf{b},$ i.e., to find all nonnegative integral vectors $\alpha \in \N^n$ satisfying the equation. This problem coincides with a fundamental problem in computational geometry: enumerating lattice points in a polyhedron $P$ in $\R^n$. These problems become equivalent when $P=P(A,\b)=\{ \alpha\in \R_{\geq 0}^n: A\alpha =\b\}$ is bounded.

For any subset $S$ in $\R^n$, denote by
$$\sigma_S(\y) = \sum_{\alpha \in S\cap \Z^n} \y^{\alpha} =\sum_{\alpha \in S\cap \Z^n} y_1^{\alpha_1}\cdots y_n^{\alpha_n} $$
its lattice point generating function. The lattice point counting problem has a nice solution
in computational geometry by Barvinok's seminal work \cite{1994A}. This work presented an algorithm that  counts lattice points in
convex rational polyhedra in polynomial time when the dimension of the polytope is
fixed. Barvinok's algorithm consists of four major steps:
0) Transform the polytope into tangent cones by Brion's theorem \cite{Brion1988}; 1) Decompose each cone into simplicial cones; 2) Decompose each simplicial cone
into unimodular cones; 3) Compute the limit of $\sigma_P(\y)$ (written as a sum of rational functions) at $y_1=\cdots=y_n=1$.
See \cite{XinTodd} for details and terminologies not presented here.

The major purpose of this paper is to give an algebraic combinatorial approach to Step 1) on simplicial cone decompositions.
Indeed, Barvinok's computation scheme can be adapted into our algebraic combinatorics framework with Step 0) avoided. See Section \ref{SimpCone} for details.

In algebraic combinatorics, we are interested in $\sigma_P(\y)$ where $P$ is a (not necessarily bounded) polyhedron represented as $P=P(A,\b)$,
which becomes a cone $K=P(A,\0)$ when $\b=\0$. This does not lose generality since
a general polyhedron can be transformed into this type.
Our development is based on the following well-known constant term formula:
\begin{align}\label{equ-inhomo}
  \sigma_{P(A,\b)} = \CT_{\Lambda} \frac{\Lambda^{-\b}}{
\prod_{j=1}^n (1- \Lambda^{A_j} y_j)}, \text{where $A_j$ is the $j$-th column  of $A$},
\end{align}
where $\CT_\Lambda$ means to take constant term in $\Lambda=(\lambda_1,\dots, \lambda_r)$
of a formal Laurent series.

Stanley used residue computation to write $ \sigma_{P(A,\b)}$ as a sum of groups, with each group a
sum of similar rational functions involving fractional powers and roots of unity. The formula
was considered too complicated to deal with but allowed Stanley to derive his Monster reciprocity theorem in \cite{1975Combinatorial}.
We observe that each group corresponds to a shifted simplicial cone. This leads to a decomposition of the following form.
\begin{align}\label{equ-dec-K}
  \sigma_{P(A,\b)}(\y)=\sum_i s_i \sigma_{K_i^{v_i}}(\y),
\end{align}
where $s_i\in \Q$, and $K_i^{v_i}$ are shifted simplicial cones.

Our major contribution is Algorithm \ref{algor-Simpcone}, which provides the decompositions described above. A simplified version of the algorithm, named \texttt{SimpCone} (for the case $\b=\0$), was first stated in \cite[Algorithm 16]{2023Algebraic} without a proof. It is used for efficient volume computation of polytopes. We will provide the promised proof in this paper.

The structure of the paper is outlined as follows. In Section \ref{sec:pre}, we introduce the concept of cones and extend the XinPF algorithm, originally developed by the first named author, to allow for constant term extraction with fractional powers. We introduce two pivotal operators that act on rational functions with fractional powers. One is \(\mathcal{Z}_{\y}\), defined in Definition \ref{defn-Z_y}, which extracts terms with integer powers of \(\y\). The other is \(\CT_{\lambda_i,j}\), defined in Definition \ref{defn-CT-Z}, which extracts the constant term in \(\lambda_i\) while considering only the contribution from the \(j\)-th denominator factor.

In Section \ref{sec:CT->Condec}, we establish the connection between constant term extraction of fractional rational functions and column operations on the corresponding matrix form. Let $I_k=\{i_1,i_2,\dots,i_k \}\subseteq[r]:=\{1,2,\dots,r\}$ and $J_k=\{j_1,j_2,\dots,j_k \}\subset [n]$ be $k$-element sets for $k\leq r$. It is known that $\CT_{\lambda_{i_{k+1}},j_{k+1}} [M_k\langle I_k;J_k\rangle]$ is a sum of similar terms involving roots of unity for $k<r$. We find that one representative can be used to reconstruct the sum. The constant term process can be performed according to the following rules.
\begin{align*}
  \CT_{\lambda_{i_{k+1}},j_{k+1}}\mathcal{Z}_{\y} [M_k\langle I_k;J_k\rangle]&=\mathcal{Z}_{\y}\CT_{\lambda_{i_{k+1}},j_{k+1}}[M_k\langle I_k;J_k\rangle],\   \  \text{(Proposition \ref{prop-commu})}\\
  \CT_{\lambda_{i_{k+1}}} [M_k\langle I_k;J_k\rangle] &= \sum_{j} \pm  \CT_{\lambda_{i_{k+1}},j} [M_k\langle I_k;J_k\rangle],\   \  \text{(Lemma \ref{lem-Irule})}\\
  \CT_{\lambda_{i_{k+1}},j_{k+1}} [M_k\langle I_k;J_k\rangle]&=\pm \mathcal{Z}_{y_{j_{k+1}}} [M_{k+1}\langle I_{k+1};J_{k+1}\rangle], \   \  \text{(Theorem \ref{theo-single})}
\end{align*}
where \(M_{k+1}\langle I_{k+1};J_{k+1}\rangle\) is obtained from \(M_k\langle I_k;J_k\rangle\) through column operations using the $(n+i_{k+1},j_{k+1})$th entry of \(M_k\langle I_k;J_k\rangle\) as the pivot.

This allows us to establish the section's main result, Theorem \ref{theo-CT-cone}, which states that each group, indexed by $J_r$, within \(\sigma_{P(A,\b)}(\y)\) corresponds to a shifted simplicial cone. The theorem is derived through a recursive process. Interestingly, the result does not depend on the recursive process but only on the set $J_r$.

Section \ref{SimpCone} is based on Theorem \ref{theo-CT-cone}. We present the \texttt{SimpCone[S]} algorithm for simplicial cone decompositions. This algorithm is versatile, offering multiple decomposition options through varying strategies \( \texttt{S}\), and it remains applicable even when the matrix \( A \) includes parameters. Such decompositions can be treated as valuable complements to geometric decompositions of cones. Our \texttt{SimpCone[S]} algorithm includes the \texttt{Polyhedral Omega} algorithm in \cite{2017Polyhedral} as a special case. Its flexibility in homogeneous cases is demonstrated with Example \ref{exam-homo}. Additionally, we provide a novel proof of Stanley's reciprocity theorem for rational cones.

Section \ref{sec:unimodular} presents the application of our framework to unimodular cone decomposition in Step 2). For a full-dimensional simplicial cone \( K \) of dimension \( d \), we can construct a homogeneous linear Diophantine system \( A\alpha = \mathbf{0} \) such that \( K \) is realized as the cone for \( M_d\langle [d];[d]\rangle \). This new viewpoint allows us to conclude that most simplicial cones (up to isomorphism) are denumerant cones, which enables the \texttt{DecDenu} algorithm (a fast algorithm for denumerant cones) to provide a unimodular cone decomposition. For all shifted cones derived from \( A\alpha = \mathbf{b} \), the \texttt{DecDenu} algorithm can be used to achieve a unimodular cone decomposition. This process is summarized in Algorithm \ref{alg_simpconeANDdecdenu}.

In Section \ref{sec:project}, we describe several future projects resulting from our framework. 
\section{Concepts and Tools}\label{sec:pre}
In the following two subsections, we introduce the concepts and tools related to cones and constant terms.
\subsection{Cones}
Let \( w_1, \dots, w_{m} \) be the column vectors of an \( n \times m \) matrix \( W \). The cone generated by these vectors is denoted as
\[ K = K(W) = K(w_1, \dots, w_m) = \left\{ \sum_{i=1}^m k_i w_i : k_i \in \mathbb{R}_{\geq 0} \right\}. \]
We assume that this set does not contain a straight line and that the generators are minimal (also known as extreme rays). A cone is termed rational if it has a set of rational generators in \( \mathbb{Q}^n \), which we also denote as \( W \). By clearing denominators, we can assume each generator \( w_i \) has integer entries. Factoring out the greatest common divisor, we obtain primitive generators \( \bar{w}_i \) such that \( \gcd(\bar{w}_i) = 1 \). The dimension \( \dim K \) of \( K \) is defined as the rank of \( W \), and \( K \) is called full-dimensional if \( \dim K = n \).

A cone \( K \) is called simplicial if the vectors \( w_1, \dots, w_m \) are linearly independent. In the full-dimensional case (i.e., when \( m = n \)), the index of \( K \) is denoted by \( \ind(K) := \lvert \det(\bar w_1, \dots, \bar w_n)\rvert \)  and the dual cone of \( K \) is known to be \( K^* = K((W^{-1})^t) \). Furthermore, a full-dimensional simplicial cone \( K \) with index \( 1 \) is called unimodular.

Let \( K^v = v + K = \{ v + \alpha : \alpha \in K \} \) denote the translation of the cone \( K \) by the vector \( v \in \mathbb{Q}^n \). This translated cone, known as a shifted cone, retains the same generators \( w_1, \dots, w_m \) but has its vertex at \( v \). For notational simplicity, we introduce the shorthand notation:
\begin{equation}\label{equ-shift}
 K^v = K^{\mydot}(w_1, w_2, \dots, w_m, v),
\end{equation}
where the superscript \(\mydot\) indicates that the cone has been shifted, with the vertex being the last column of the matrix \( (W,v) \).

Geometers have developed a variety of methods for triangulating a cone \( K \) into simplicial cones, including the methods of Lasserre \cite{1983An} and Lawrence \cite{1991Polytope}. One such triangulation yields an expression of the form through shifting:
\[ \sigma_{K^v}(\mathbf{y}) = \sum_i s_i \sigma_{K_i^v}(\mathbf{y}), \]
where \( s_i = \pm 1 \) and \( K_i \) are the resulting simplicial cones. This contrasts with the decomposition presented in Equation \eqref{equ-dec-K}, where the simplicial cones are translated by distinct vertices, and the coefficients \( s_i \) can be rational numbers.

\subsection{Brief introduction to the XinPF Algorithm}
We briefly introduce Xin's partial fraction algorithm (XinPF for short) on which the \texttt{SimpCone[S]} algorithm is built. The new point is that
we allow fractional powers. In what follows, the term fractional always refers to fractional powers and $N$ is a suitable large integer.
For instance, a fractional Laurent polynomial in $y_i, i=1,\dots, n$ is just a Laurent polynomial in $y_i^{1/N}$ for all $i$.

We often use the following notation. For a nonzero rational function $F(\lambda)$, define $\deg_\lambda F(\lambda)$ to be the degree of the numerator minus the degree of the denominator, and define $\ord_{\lambda}F(\lambda)$ to be the pole order of $F(\lambda)$, that is, the $p$ such that
$\lim_{\lambda=0} \lambda^p F(\lambda)$ is a nonzero constant. If $\deg_\lambda F(\lambda)<0$, then we say $F(\lambda)$ is a proper rational function.
If $\ord_{\lambda}F(\lambda)< 0$ then $F(0)=0$. These notation naturally extend to the fractional case.

We use the list $\texttt{vars}=[y_1,y_2,\dots, y_n]$ to specify the order of the variables, which defines the working field
 $G=\mathbb{C}((y_n^{1/N}))((y_{n-1}^{1/N}))\cdots ((y_1^{1/N}))$.
$G$ is called a field of iterated Laurent series in \cite{xin2004fast}, except that we allow fractional powers here.
An element $\eta$ in the field $G$ is called a fractional series (short for fractional iterated Laurent series).
In particular, if $\eta$ has only finitely many nonzero coefficients, then we call it a fractional Laurent polynomial.
We only need the following properties of $G$. For more details, interested readers are referred to \cite{xin2004fast}.

\emph{Elements in $G$ have unique series expansions.}
Let $\alpha\in \Q^n$ be a nonzero vector with the first nonzero entry $\alpha_j$, so $\alpha_i=0$ for $i<j$.
A ``monomial" $M=\zeta \y^\alpha$, where $\zeta \in \C$ is usually a root of unity,
is said to be \emph{small} if $\alpha_j >0$, denoted $M<1$;
and is said to be \emph{large} if otherwise $\alpha_j<0$, denoted $M>1$.
By working in $G$, we only need to keep in mind the following two series expansions:
\begin{align} \label{issmall}
  \frac{1}{1-M} =\left\{
                   \begin{array}{ll}
                    \displaystyle \sum_{k\ge 0} M^k, & \text{ if } M<1; \\
                    \displaystyle \frac{1}{-M(1-1/M)}=  \sum_{k\ge 0} - \frac{1}{M^{k+1}}, & \text{ if } M>1.
                   \end{array}
                 \right.
\end{align}

An Elliott-rational function $E$ is a special type of rational function.
When expanding $E$ as a series in $G$, we usually rewrite $E$ in its \emph{proper form}:
\begin{align}\label{e-properform}
  E = \frac{L}{(1-M_1)(1-M_2)\cdots (1-M_n)},  \tag{proper form}
\end{align}
where $L$ is a Laurent polynomial and $M_i<1$ for all $i$. Note that the proper form of $E$ is not unique. For instance, both $1/(1-y)$ and $(1+y)/(1-y^2)$ are proper forms.

Now we work in the field $G$ specified by the list $\texttt{vars}=[\y,\Lambda]=[y_1,y_2,\dots, y_n,\lambda_1,\dots, \lambda_r]$.
For any $F$ in $G$, define the constant term operator by
$$\CT_\Lambda F = \CT_\Lambda \sum_{} f_{i_1,\dots,i_r} \lambda_1^{i_1} \cdots \lambda_r^{i_r}  = f_{0,\dots,0}.$$
This is well-defined as $F$ has a unique series expansion in $G$. The constant term operators $\CT_{\lambda_i}$ defined this way commute with each other, so $\CT_{\Lambda}$ extracts the constant term in a set of variables.
The core problem is to compute the constant term of an Elliott-rational function $E$, which is known to be another Elliott-rational function.

Algorithm XinPF is an elimination-based algorithm, so the basic problem is to compute
$\CT_\lambda E$ in $G$ for a variable $\lambda=\lambda_i$. Since the substitution $\lambda\to \lambda^m$
does not change the constant term for any positive integer $m$, we assume $E$ is a rational function in $\lambda$.
For this constant term in a single variable, we perform partial fraction decomposition and extract the constant term.

To be precise, we first write $E$ as a standard rational function in $\lambda$:
\begin{align}\label{e-positive-form}
E= \frac{L(\lambda)}{\prod_{i=1}^n (1-u_i \lambda^{a_i})},
\end{align}
where $L(\lambda)$ is a Laurent polynomial, $u_i$ are independent of $\lambda$ and $a_i$ are
positive integers for all $i$. Note that $u_i \lambda^{a_i}$ might be large, like $y_1^{-1}\lambda^2>1$.

\begin{prop}\label{p-partialfraction}\cite{2015A}
Suppose the partial fraction decomposition of $E$ is given by
\begin{align}
  \label{e-E-parfrac}
E= P(\lambda)+ \frac{p(\lambda)}{\lambda^k} +\sum_{i=1}^n \frac{A_i(\lambda)}{1-u_i \lambda^{a_i}},
\end{align}
where the $u_i$'s are independent of $\lambda$, $P(\lambda)$, $p(\lambda)$, and the $A_i(\lambda)$'s are all polynomials, $\deg p(\lambda)<k$, and $\deg A_i(\lambda)<a_i$ for all $i$.
Then we have
\begin{equation*}
\CT_\lambda E = P(0) + \sum_{u_i \lambda^{a_i} <1} A_i(0).
\end{equation*}
Moreover, if $E$ is \emph{proper in $\lambda$}, i.e., $\deg_\lambda E<0$, then
\begin{equation}\label{equ-contri}
  \CT_\lambda E = \sum_{u_i \lambda^{a_i} <1} A_i(0)
\end{equation}

If $E|_{\lambda=0}=\lim_{\lambda\rightarrow 0} E$ exists, i.e., $\ord_\lambda E\leq 0$, then
\begin{equation}\label{equ-dual-contri}
  \CT_\lambda E = E|_{\lambda=0} - \sum_{u_i \lambda^{a_i} >1} A_i(0).
\end{equation}
\end{prop}
The proposition holds since, if written in proper form, we obtain
$$ \frac{A_i(\lambda)}{1-u_i \lambda^{a_i}} =\left\{\begin{array}{ll}
                     \displaystyle \frac{A_i(\lambda)}{1-u_i \lambda^{a_i}} \ \mathop{\longrightarrow}\limits^{\CT_\lambda} A_i(0), & \text{ if } u_i \lambda^{a_i} <1; \\
                    \displaystyle \frac{A_i(\lambda)}{-u_i\lambda^{a_i} (1-\frac{1}{u_i \lambda^{a_i}})}=\frac{\lambda^{-a_i}A_i(\lambda)}{-u_i(1-\frac{1}{u_i \lambda^{a_i}})}
                    \ \mathop{\longrightarrow}\limits^{\CT_\lambda} 0, & \text{ if } u_i \lambda^{a_i}>1.
                   \end{array}
                 \right.
$$
Formula \eqref{equ-dual-contri} is the dual of \eqref{equ-contri}.
Because of these two formulas, it is convenient to call the denominator factor $1-u_i\lambda^{a_i}$
\emph{contributing} if $u_i\lambda^{a_i}< 1$ and \emph{dually contributing} if $u_i\lambda^{a_i}>1$.

\begin{defn}\label{defn-A}
Let $E$ be as above. We denote by
$$\CT_\lambda \frac{1}{\underline{1-u_s \lambda^{a_s}}}\cdot E (1-u_s \lambda^{a_s})=\CT_\lambda \frac{1}{\underline{1-(u_s \lambda^{a_s})^{-1}}}\cdot    E(1-(u_s \lambda^{a_s})^{-1}) :=A_s(0).$$
\end{defn}

By allowing fractional powers and using roots of unity, we have
\begin{equation}\label{equ-A_s(0)}
\frac{A_s(\lambda)}{1-u_s \lambda^{a_s}}=\frac{1}{a_s}\sum_{\zeta: \zeta^{a_s}=1} \frac{1}{1-u_s^{\frac{1}{a_s}}\zeta\lambda}   \big(E\cdot(1-u_s \lambda^{a_s})\big)\Big{|}_{\lambda= (u_s^{\frac{1}{a_s}}\zeta)^{-1}}.
\end{equation}

This provides an explicit formula for $A_s(0)$ as follows.
\begin{lem}\label{lem-unit-E}
Suppose $E$ as in \eqref{e-positive-form} is rewritten in the form
$$E= \frac{L(\lambda)}{\prod_{i=1}^n (1-u_i \lambda^{b_i})},$$
where $b_i=\pm a_i$ for all $i$. Then for nonzero integer $b_s$,
\begin{equation}\label{equ-addifomula-A}
\CT_\lambda \frac{1}{\underline{1-u_s \lambda^{b_s}}}\cdot E (1-u_s \lambda^{b_s})=\frac{1}{b_s}\sum_{\zeta: \zeta^{b_s}=1}E (1-u_s \lambda^{b_s}) \Big|_{\lambda=({u_j}^{\frac{1}{b_s}}\zeta)^{-1}}.
\end{equation}
\end{lem}
\begin{proof}
When $b_s=a_s>0$, the formula is a direct consequence of \eqref{equ-A_s(0)}.

When $b_s=-a_s<0$, we have to rewrite $E (1-u_s \lambda^{b_s})$, denoted as  $F$, to apply \eqref{equ-A_s(0)}:
\begin{align*}
  A_s(0)&=\CT_\lambda \frac{1}{\underline{1-u_s \lambda^{b_s}}}\cdot F=\CT_\lambda \frac{1}{\underline{1-u_s^{-1} \lambda^{a_s}}}\cdot \big(-u_s^{-1}\lambda^{a_s}\cdot F\big)\\
  &=\frac{1}{a_s}\sum_{\zeta: \zeta^{a_s}=1}\big(-u_s^{-1}\lambda^{a_s}\cdot F\big)\Big|_{\lambda=({u_s}^{-\frac{1}{a_s}}\zeta)^{-1}}
  =\frac{1}{b_s}\sum_{\zeta: \zeta^{b_s}=1}F\Big|_{\lambda=({u_s}^{\frac{1}{b_s}}\zeta)^{-1}}.
\end{align*}
The lemma then follows.
\end{proof}

Thus, when taking the constant term in $\Lambda$, the iterative application of Lemma \ref{lem-unit-E} will result in groups of similar terms involving roots of unity.
It is not easy to remove these roots of unity. The key observation is that we need not go into details of each group:
one natural representative is sufficient to reconstruct the whole group of terms
by using the operators $\mathcal{Z}_{y_j}$ defined as follows.

\begin{defn}\label{defn-Z_y}
Define the operators $\mathcal{Z}_{y_j}$ acting on fractional series by
\begin{equation}
  \mathcal{Z}_{y_j}\sum_{\kappa\in\Q^n}a_{\kappa}\y^{\kappa}:=\sum_{\kappa\in\Q^{n},\kappa_j\in \Z}a_{\kappa}\y^{\kappa},
\end{equation}
where the coefficients $a_{\kappa}$ are free of the $y$ variables.
In other words, $\mathcal{Z}_{y_j}$ extracts all terms with integer exponents of $y_j$.
\end{defn}

Let $J$ be a $k$-subset of $[n]$. We will consider a rational function
 $E\langle J\rangle$ of the form
\begin{align}\label{e-integer-form}
E\langle J\rangle= \frac{L(\lambda)}{\prod_{j=1}^n (1-u_j \lambda^{b_j})},
\end{align}
where $b_j \in \Q$ and the denominator factors are coprime to each other.
The notation $\langle J\rangle$ indicates that $u_j=0$ for $j\in J$, and thus the $j$-th denominator factor is replaced by $1$ (we treat $b_j=0$).

\begin{defn}\label{defn-CT-Z}
Suppose $E\langle J\rangle$ is given as in \eqref{e-integer-form}.
If $b_j\in \Q$ is nonzero and the variable $y$ only appears at $u_j=yu_j'$, then define
$$ \CT_{\lambda, j} E\langle J\rangle={\CT_\lambda}^{(y)} \frac{1}{\underline{1-u_j \lambda^{b_j}}} E\langle J\rangle(1-u_j \lambda^{b_j}):=\sgn(b_j) \mathcal{Z}_{y} \left(E\langle J\rangle (1-u_j \lambda^{b_j}) \right)\Big|_{\lambda= (yu_j')^{-1/{b_j}}}.$$
In other words, it is obtained, up to a sign, by applying the $\mathcal{Z}_{y}$ operator to a fractional rational function of the form $E\langle J\cup \{j\}\rangle$,
which comes from $E\langle J\rangle$ by first replacing $y$ with $0$, then replacing $\lambda$ with $(yu_j')^{-1/{b_j}}$.
\end{defn}
Note: The first notation depends on the order of the denominator factors, while the second (underlined) notation does not.
See the following example which clarifies the notation.
\begin{exam}
Let $E\langle \{2\}\rangle=\frac{\lambda_1^{3}\lambda_2}{(1-\lambda_1^{-2}\lambda_2^{-2}y_1)(1)(1-\lambda_1^{3}\lambda_2y_3)(1-\lambda_1y_4)}$. Then we have
\begin{align*}
 \CT\limits_{\lambda_1,3}E\langle \{2\}\rangle &={\CT_{\lambda_1}}^{(y_3)} \frac{\lambda_1^{3}\lambda_2}{(1-\lambda_1^{-2}\lambda_2^{-2}y_1)(1)\underline{(1-\lambda_1^{3}\lambda_2y_3)}(1-\lambda_1y_4)}\\
 &={\CT_{\lambda_1}}^{(y_3)}  \frac{\lambda_1^{3}\lambda_2}{(1-\lambda_1^{-2}\lambda_2^{-2}y_1)\underline{(1-\lambda_1^{3}\lambda_2y_3)}(1)(1-\lambda_1y_4)}=\CT\limits_{\lambda_1,2}E'\langle\{3\}\rangle,
\end{align*}
where $E'$ is obtained from $E$ by exchanging the second and third denominator factor.
\end{exam}

Now we show that Definition \ref{defn-CT-Z} naturally extends Definition \ref{defn-A} in the following sense.
\begin{prop}\label{prop-singleCT}
Suppose $E\langle J\rangle$ is given as in \eqref{e-integer-form} with $y$ appearing only at $u_j=yu_j'$. If $E\langle J\rangle \big|_{\lambda\to \lambda^m}$ is integral in $\lambda$ for a positive integer  $m\in \mathbb{P}$,
then $$\CT_{\lambda, j} E\langle J\rangle \big|_{\lambda\to \lambda^m} = \CT_{\lambda, j} E\langle J\rangle$$
is independent of $m$.

Moreover, if $E\langle J\rangle $ is integral in $\lambda$, then
$$\CT_\lambda \frac{1}{\underline{1-u_j \lambda^{b_j}}} E\langle J\rangle(1-u_j \lambda^{b_j})=\CT_{\lambda, j}E\langle J\rangle.$$
\end{prop}
\begin{proof}
The first part is checked directly as follows:
\begin{align*}
\CT_{\lambda, j} E\langle J\rangle\big|_{\lambda\to \lambda^m} &= \sgn(mb_j)\mathcal{Z}_{y} \left(E\langle J\rangle \big|_{\lambda= \lambda^m} (1-u_j'y \lambda^{mb_j}) \right)\Big|_{\lambda =  (yu_j')^{-\frac{1}{mb_j}}}\\
&=\sgn(b_j)\mathcal{Z}_{y} \left(E\langle J\rangle (1-u_j \lambda^{b_j}) \right)\Big|_{\lambda=  (yu_j')^{-\frac{1}{b_j}}}.
\end{align*}

For the second part,  denote $E\langle J\rangle \cdot (1-u_j \lambda^{b_j})$ by $H(\lambda)$ for simplicity.
Expand $H(\lambda)$ as a series $\sum_{\kappa \in \mathbb{Z}} h_{\kappa} \lambda^{\kappa}$, where $h_{\kappa}$ are independent of $\lambda$ and $y$. By applying \eqref{equ-addifomula-A},
we can write the left-hand side as
\begin{align*}
  A_j(0)&
 =\frac{1}{b_j} \sum_{\zeta: \zeta^{b_j}=1}  \left(\sum_{\kappa\in\Z}h_{\kappa}\lambda^{\kappa} \right)\Bigg|_{\lambda=((yu')^{\frac{1}{b_j}}\zeta)^{-1}}
 =\frac{1}{b_j} \sum_{\zeta: \zeta^{b_j}=1} \sum_{\kappa\in\Z}h_{\kappa}((yu')^{\frac{1}{b_j}}\zeta)^{-\kappa}\\
&=\sum_{\kappa\in\Z}h_{\kappa} (yu_j')^{-\kappa/{b_j}}\cdot\frac{1}{b_j}\sum_{\zeta: \zeta^{b_j}=1}\zeta^{-\kappa}=\sgn(b_j)\sum_{\kappa\in\Z}\chi(b_j|\kappa)\cdot h_{\kappa} (yu_j')^{-\frac{\kappa}{b_j}},
\end{align*}
where the last equality follows from the fact that $\frac{1}{b_j} \sum_{\zeta: \zeta^{b_j} = 1} \zeta^{-\kappa} = \sgn(b_j) \cdot \chi(b_j|\kappa)$. Here, for any statement $D$, $\chi(D)$ is 1 if $D$ is true and 0 if $D$ is false.

Thus, we get
$$A_j(0)=\sgn(b_j)\mathcal{Z}_{y}H(\lambda)\Big|_{\lambda= (yu_j')^{-\frac{1}{b_j}}}=\CT_{\lambda, j} E\langle J\rangle,$$
as desired.
\end{proof}
Next, we show that Proposition \ref{p-partialfraction} naturally extends as follows.
\begin{cor}\label{cor-contri-dual}
Suppose $E\langle J\rangle$ is a fractional rational function given as in \eqref{e-integer-form}  where $y_j$ appears only at $u_j=y_ju_j'$.
If $E\langle J\rangle$ is \emph{proper in $\lambda$}, then
\begin{equation}\label{e-contri}
  \CT_\lambda E\langle J\rangle = \sum_{u_j \lambda^{b_j} <1}\CT_{\lambda, j} E\langle J\rangle;
\end{equation}

If $E\langle J\rangle|_{\lambda=0}=\lim_{\lambda\rightarrow 0}$ exists, i.e., $E\langle J\rangle$ has a non-positive pole order in $\lambda$, then
\begin{equation}\label{e-dual-contri}
  \CT_\lambda E\langle J\rangle
=E\langle J\rangle\big|_{\lambda=0} - \sum_{u_j \lambda^{b_j}>1}\CT_{\lambda, j} E\langle J\rangle.
\end{equation}
\end{cor}
\begin{proof}
Choose an $m\in \mathbb{P}$ such that $E\langle J\rangle \big|_{\lambda\to \lambda^m}$ is integral in $\lambda$. Denote this rational function by $E'\langle J\rangle$. Apply Proposition \ref{prop-singleCT} with respect to $E'\langle J\rangle$. It follows that
$\CT_{\lambda, j} E'\langle J\rangle$ naturally corresponds to the contribution of the $j$-th denominator factor in $E'\langle J\rangle$.
Applying Proposition \ref{p-partialfraction} gives
\begin{align*}
\CT_\lambda E\langle J\rangle=\CT_\lambda E'\langle J\rangle=
\begin{cases}
   \sum_{u_j \lambda^{b_j} <1}\CT_{\lambda, j} E'\langle J\rangle, & \mbox{if $E'\langle J\rangle$ is \emph{proper in $\lambda$}}; \\
   E'\langle J\rangle|_{\lambda=0}-\sum_{u_j \lambda^{b_j} >1}\CT_{\lambda, j} E'\langle J\rangle, &\mbox{if  $E'\langle J\rangle|_{\lambda=0}$ \emph{exists}}.
\end{cases}
\end{align*}
The corollary then follows since $\CT_{\lambda, j} E'\langle J\rangle=\CT_{\lambda, j} E\langle J\rangle$ by
Proposition \ref{prop-singleCT}.
\end{proof}

We also need some properties of the $\mathcal{Z}_{y_i}$ operators for later sections.
In the definition, $\mathcal{Z}_{y_i}$ commutes with $\mathcal{Z}_{y_j}$ so that $\mathcal{Z}_{\y}$ acts on a set of variables.

It is easy to see that $\mathcal{Z}_{y_j}$ is linear in the following sense.
\begin{itemize}
  \item [1.] For fractional series $\{F_i(\y)\}_{1\leq i \leq T}$, we have $\mathcal{Z}_{y_j}\sum_{i=1}^T F_i(\y)=\sum_{k=1}^T\mathcal{Z}_{y_j}F_i(\y)$.

  \item [2.] If the fractional series $P(\y)$ contains only integer exponents of $y_j$, then we have
  $$\mathcal{Z}_{y_j}(P(\y)\cdot F(\y))=P(\y)\cdot\mathcal{Z}_{y_j}F(\y).$$
  \end{itemize}

The following proposition describes the action of the operator $\mathcal{Z}_{\y}$ on Elliott-rational functions.
\begin{prop}\label{prop-Z_y}
The operator $\mathcal{Z}_{\y}$ maps a fractional Elliott-rational function to another Elliott-rational function.
The result can be obtained through a finite number of steps and is \textrm{independent} of the series expansion of the rational function.
\end{prop}
\begin{proof}
Let $F(\y)$ be a fractional Elliott-rational function. It can be written as
\begin{equation*}
F(\y)=\frac{L(\y)}{\prod_{i=1}^{n}(1-u_i\y^{c_i})},
\end{equation*}
where $c_i \in \Q^n$, $u_i$ is independent of $\y$, and $L(\y)$ is a fractional Laurent polynomial. Such $F(\y)$ is identified with its unique series expansion in the field $G$,
so that $\mathcal{Z}_{\y}$ acts on $F(\y)$.

We rewrite $F(\y)$ as
$$L(\y)\cdot\prod_{i=1}^{n}\frac{\sum_{j=0}^{p_i-1}(u_i\y^{c_i})^{j}}{1-(u_i\y^{c_i})^{p_i}},$$
where $p_i$ are positive integers such that $p_ic_i$ are integral for all $i$. We can further assume the entries of $p_ic_i$ have the greatest common divisor $1$ for each $i$.

In the field $G$, the series expansion of
$(1-(u_i\y^{c_i})^{p_i})^{-1}$ contains only integer powers of $\y$, regardless of whether $(u_i\y^{c_i})^{p_i}$ is large or small in $G$. Thus, by
the linear property of $\mathcal{Z}_{\y}$, we have
$$\mathcal{Z}_{\y}F(\y)=\mathcal{Z}_{\y}\Big(L(\y)\cdot\prod_{i=1}^{n}\frac{\sum_{j=0}^{p_i-1}(u_i\y^{c_i})^{j}}{1-(u_i\y^{c_i})^{p_i}}\Big)=\frac{\mathcal{Z}_{\y}\big(L(\y)\cdot\prod_{i=1}^{n}\big(\sum_{j=0}^{p_i-1}(u_i\y^{c_i})^{j}\big)\big)}{\prod_{i=1}^{n}(1-(u_i\y^{c_i})^{p_i})}.$$
This is an Elliott rational function since the numerator is obtained by applying $\mathcal{Z}_{\y}$ to a fractional Laurent polynomial,
which takes only a finite number of steps and results in a Laurent polynomial.

Finally, observe that the resulting Elliott rational function is independent of its series expansion in $G$. This completes the proof.
\end{proof}

\section{Cone Decomposition from Constant Term Extraction}\label{sec:CT->Condec}
In this section, we establish the relationship between specific constant terms and cones.
Our starting point is the constant term formula defined in \eqref{equ-inhomo}.
By iterative application of Proposition \ref{prop-singleCT}, we obtain a sum of the form $\sum_i s_i R_i,$
where $s_i\in \Z$ and each $R_i$ is a sum of similar rational functions involving fractional powers and several roots of unity.
Dealing with these roots of unity was considered challenging \cite[Remark 4.3]{2007A}. We show that
$R_i$ is associated with a nonzero $r\times r$ minor of $A$ and corresponds to a shifted simplicial cone.

The single equation case has been addressed in \cite{xin2024polynomial}. Here, we will give a rigourous proof of the general case.
\subsection{Matrix Forms}
We use a slight variation of Stanley's notation. We start with
\begin{equation}\label{equ-M}
M=\left(\begin{array}{cc}
         \textit{id}_n& \0\\
         A            &-\b
        \end{array}\right)=(m_{i,j})_{(n+r)\times {(n+1)}}, \text{~where $\textit{id}_n$ is the identity matrix of order $n$}.
\end{equation}
This matrix encodes the rational function in \eqref{equ-inhomo}.

We successively perform \emph{elementary operations} as follows.
For $1\leq j\leq n$, we define $M\leftarrow (i,j)$ to be the matrix obtained from $M$ by Gaussian column elimination
with the $(n+i,j)$th entry of $M$ as the \emph{pivot}, i.e., adding
$-\frac{m_{n+i,s}}{m_{n+i,j}}$ multiple of the $j$th column to the $s$th column for each $s\neq j$. We call the nonzero entry $m_{n+i,j}$ the \emph{pivot item}, and we will ignore row $n+i$ and column $j$
in further operations.
More generally, recursively define
$$M\leftarrow ((i_1,\dots, i_k);(j_1,\dots, j_k))= M \leftarrow (i_1,j_1)\leftarrow \cdots \leftarrow (i_k,j_k).$$
The pivot (item) sequence $(p_1,p_2,\dots, p_k)$ records the pivot items in the above process, i.e., $p_\ell$ is the
$(n+i_{\ell},j_\ell)$ entry of $M\leftarrow ((i_1,\dots, i_{\ell-1});(j_1,\dots, j_{\ell-1}))$. Note that $p_1\cdots p_k$ is the
corresponding $k\times k$ minor of $A$, which is denoted by $A\begin{pmatrix}
i_1,\dots,i_{k}\\
j_1,\dots,j_{k}
\end{pmatrix}$. In other words, it is the determinant of the matrix whose $(s,t)$ entry is the $(i_s,j_t)$ entry of $A$.

The next result simplifies the notation.
\begin{prop}\label{theo-change}
Suppose $M=\begin{pmatrix}
         \textit{id}_n& \0\\
         A            &-\b
        \end{pmatrix}$. Then for any $k\in[r]$, the matrix
        $$M\leftarrow ((i_1,\dots, i_k);(j_1,\dots, j_k))$$
has its first $n-k$ column vectors linearly independent.

Moreover, for any
permutations $\pi$ and $\pi'$ on $1,2,\dots, k$, we have
 $$M\leftarrow ((i_1,\dots, i_k);(j_1,\dots, j_k))= M\leftarrow ((i_{\pi(1)},\dots, i_{\pi(k)});(j_{\pi'(1)},\dots, j_{\pi'(k)})),$$
if both sides are well-defined and understood as ignored rows and columns removed.
\end{prop}
\begin{proof}
  Without loss of generality, we may assume that $(i_s,j_s)=(s,s)$, that is, to prove the first $n-k$ column vectors of $M\leftarrow ([k];[k])$ are
   linearly independent and
$$M\leftarrow ((1,\dots, k);(1,\dots, k))=M\leftarrow((\pi(1),\dots, \pi(k));(\pi'(1),\dots, \pi'(k))).$$

Write $A=
\begin{pmatrix}
A_1&A_2\\
A_3&A_4\\
\end{pmatrix}$, where $A_1\in \Z^{k\times k}$ is nonsingular.
Then $M=\begin{pmatrix}
\textit{id}_k &\0      &\0 \\
\0   & \textit{id}_{n-k}&\0\\
A_1 &   A_2  &-\b_1\\
A_3 &A_4&-\b_2
\end{pmatrix}$
and after column operations,  $A_2$ and $\b_1$ become $\0$.
The column operations are equivalent to adding the $U$ (right) multiple of the left block
to the middle block and adding the $\mathbf{q}$ (right) multiple of the left block to the right block, where $U$ is a matrix such that $A_2+A_1 U$ is the zero matrix and $\mathbf{q}$ is a column vector such that $-\b_1+A_1\mathbf{q}$ is a zero vector.
It follows that $U=-{A_1}^{-1}A_2,~\mathbf{q}={A_1}^{-1}\b_1$ and
\begin{equation}\label{equ-Mk}
M\leftarrow(   (\pi(1),\dots, \pi(k) );( \pi'(1),\dots, \pi'(k) )  )=\begin{pmatrix}
-{A_1}^{-1}A_2&{A_1}^{-1}\b_1\\
\textit{id}_{n-k}&\0\\
A_4-A_3{A_1}^{-1}A_2&A_3A_1^{-1}\b_1-\b_2
\end{pmatrix}
\end{equation}
is independent of $\pi$ and $\pi'$.
The column vectors (excluding the last column) are linearly independent due to the $\textit{id}_{n-k}$.
This completes the proof.
\end{proof}

For convenience, we denote the matrix $M\leftarrow ((i_1,\dots, i_k);(j_1,\dots, j_k))$ by $M_k\langle I;J\rangle$, where we use $\langle I;J\rangle=\langle\{i_1,\dots, i_k\};\{j_1,\dots,j_k\}\rangle $ to emphasize that
rows $n+I$ and columns $J$ of $M_k\langle I;J\rangle$ are ignored.

Next, we define
\begin{equation}\label{equ-Psi}
[M_k\langle I;J\rangle]=\frac{(\y,\Lambda)^{\gamma_{n+1}}}{(1-(\y,\Lambda)^{\gamma_1})\cdots (1-(\y,\Lambda)^{\gamma_n})},
\end{equation}
where the $(n+1)$th column $\gamma_{n+1} \in \Q^{n+r}$ of $M_k\langle I;J\rangle$ is converted to the numerator $(\y,\Lambda)^{\gamma_{n+1}}$ and the $j$th ($1\leq j\leq n$) column $\gamma_j \in \Q^{n+r}$ of $M_k\langle I;J\rangle$ is converted to the $j$th ($1\leq j\leq n$) denominator factor according to whether
$j$ belongs to $J$ or not: the factor is $1-(\y,\Lambda)^{\gamma_j}$ if
$j\not\in J$ and is $1$ if $j\in J$.
The $M_k\langle I;J\rangle$ and $[M_k\langle I;J\rangle]$ will be referred to as the matrix form and the Erat form, respectively.
It should be clear that $[M_k\langle I;J\rangle]$ is free of $\lambda_i, \ i\in I$.
In particular, if $k=r$ then $[M_r\langle [r];J\rangle]$ is free of all the $\lambda_i$'s,
so we shall treat $M_r\langle J\rangle$ (short for $M_r\langle [r];J\rangle$) as rows $n+[r]$ and columns $J$ removed, and it will be seen to correspond to
an $n-r$ dimensional shifted simplicial cone in $\R^n$.

\subsection{Constant term extraction over matrix forms}
We first prove the commutativity between the operations $\CT_{\lambda_i,j}$ and $\mathcal{Z}_{\y}$.
 \begin{prop}\label{prop-commu}
 Follow the notation as above. If
$i \not\in I,~j\not\in J$, then we have
 \begin{equation}
\CT_{\lambda_i,j}\mathcal{Z}_{\y} [M_k\langle I;J\rangle]=\mathcal{Z}_{\y}\CT_{\lambda_i,j}[M_k\langle I;J\rangle].
\end{equation}
\end{prop}
\begin{proof}
By renaming the $\lambda$ variables, we may assume $I=[k]$ and $i=k+1$. By permuting the denominator factors and
renaming the $y$ variables, we may assume $J=[k]$ and $j=k+1$. For simplicity,  we redefine the variable as $k\to k-1$. Thus it suffices to prove
$$\CT_{\lambda_{k},k}\mathcal{Z}_{\y} [M_{k-1}\langle [k-1];[k-1]\rangle]=\mathcal{Z}_{\y}\CT_{\lambda_{k},k}[M_{k-1}\langle [k-1];[k-1]\rangle].$$

By Equation \eqref{equ-Psi} and \eqref{e-integer-form} and by choosing a suitable $m\in \mathbb{P}$ such that
\begin{equation}\label{equ-E-M}
E_{k-1}\langle [k-1]\rangle=[M_{k-1}\langle [k-1];[k-1]\rangle] \Big|_{\lambda_{k}\to \lambda_{k}^m} =\frac{\lambda_{k}^{q_{n+1}}(\hat{\y},\hat{\Lambda})^{Q_{n+1}}}{\prod\limits_{j=1}^{n}\big(1-y_{j}\lambda_{k}^{q_j}(\hat{\y},\hat{\Lambda})^{Q_{j}}\big)} ~(q_k\neq 0),
\end{equation}
where $q_j\in\Z$ for $j\geq k$,  $\hat{\Lambda}$ is free of $\lambda_i$ for $i\in [k]$, and $\hat{\y}=(y_1,y_2,\dots,y_{k-1})$.

Then we rewrite  $E_{k-1}\langle [k-1] \rangle$ as a standard rational function
\begin{equation*}
  \label{e-vol-simplical-cone}
\frac{L(\lambda_k)}{\prod\limits_{j=1}^{n}\big(1-y_{j}^{\sgn(q_j)}\lambda_{k}^{\bar{q}_j}(\hat{\y},\hat{\Lambda})^{\bar{Q}_{j}}\big)}
\end{equation*}
where $L(\lambda_k)$ is a Laurent polynomial and $\bar{q}_j=\sgn(q_j)q_j,~\bar{Q}_j=\sgn(q_j)Q_j$ for all $j$.

By partial fraction decomposition as in Equation \eqref{e-E-parfrac}, we have
$$E_{k-1}\langle [k-1] \rangle=P(\lambda_k)+\frac{p(\lambda_k)}{\lambda_k^t}+\sum_{j=k}^{n}\frac{A_j(\lambda_k)}{1-y_{j}^{\sgn(q_j)}\lambda_{k}^{\bar{q}_j}(\hat{\y},\hat{\Lambda})^{\bar{Q}_{j}}}.$$

Consequently by the linear property of $\mathcal{Z}_{\y}$, we obtain
\begin{equation}\label{equ-PFD}
\mathcal{Z}_{\y}E_{k-1}\langle [k-1] \rangle
=\mathcal{Z}_{\y}P(\lambda_k)+\mathcal{Z}_{\y}\frac{p(\lambda_k)}{\lambda_k^t}+\sum_{j=k}^{n}\mathcal{Z}_{\y}\frac{A_j(\lambda_k)}{1-y_{j}^{\sgn(q_j)}\lambda_{k}^{\bar{q}_j}(\hat{\y},\hat{\Lambda})^{\bar{Q}_{j}}}.
\end{equation}

On the other hand, by Proposition \ref{prop-Z_y}, we have
\begin{equation*}
 \mathcal{Z}_{\y}E_{k-1}\langle [k-1] \rangle=\frac{\mathcal{Z}_{\y}\left(L(\lambda_k)\cdot\prod_{j=1}^{n}\left(\sum_{i=0}^{p_j-1}\big(y_{j}^{\sgn(q_j)}\lambda_{k}^{\bar{q}_j}(\hat{\y},\hat{\Lambda})^{\bar{Q}_{j}}\big)^{i}\right) \right)}{\prod\limits_{j=1}^{n}\left(1-\big(y_{j}^{\sgn(q_j)}\lambda_{k}^{\bar{q}_j}(\hat{\y},\hat{\Lambda})^{\bar{Q}_{j}}\big)^{p_j}\right)}
\end{equation*}
and
\begin{equation*}
 \mathcal{Z}_{\y}\frac{A_j(\lambda_k)}{1-y_{j}^{\sgn(q_j)}\lambda_{k}^{\bar{q}_j}(\hat{\y},\hat{\Lambda})^{\bar{Q}_{j}}}=\frac{\mathcal{Z}_{\y}\left(A_j(\lambda_k)\cdot\left(\sum_{i=0}^{p_j-1}\big(y_{j}^{\sgn(q_j)}\lambda_{k}^{\bar{q}_j}(\hat{\y},\hat{\Lambda})^{\bar{Q}_{j}}\big)^{i}\right) \right)}{1-\big(y_{j}^{\sgn(q_j)}\lambda_{k}^{\bar{q}_j}(\hat{\y},\hat{\Lambda})^{\bar{Q}_{j}}\big)^{p_j}}.
\end{equation*}

It is easy to see that \eqref{equ-PFD} does give the partial fraction decomposition of $\mathcal{Z}_{\y}E_{k-1}\langle [k-1] \rangle$.

According to Proposition \ref{prop-singleCT} and Equation \eqref{equ-E-M}, we obtain
\begin{align*}
\CT\limits_{\lambda_{k},k}\mathcal{Z}_{\y}[M_{k-1}\langle [k-1];[k-1]\rangle]&= \CT_{\lambda_k,k}\mathcal{Z}_{\y}E_{k-1}\langle [k-1] \rangle\\
&= \mathcal{Z}_{\y}\left(\sum_{i=0}^{p_k-1}A_k(\lambda_k)\cdot\big(y_{k}^{\sgn(q_k)}\lambda_{k}^{\bar{q}_k}(\hat{\y},\hat{\Lambda})^{\bar{Q}_{k}}\big)^{i} \right)\Bigg|_{\lambda_k=0} \\
&= \mathcal{Z}_{\y} A_k(0)=\mathcal{Z}_{\y}\CT\limits_{\lambda_{k},k}E_{k-1}\langle [k-1] \rangle\\
&=\mathcal{Z}_{\y}\CT\limits_{\lambda_{k},k}[M_{k-1}\langle [k-1];[k-1]\rangle].
\end{align*}
This completes the proof.
\end{proof}


When the operator \(\CT_{\lambda_{i,j}}\) is applied to \([M_k\langle I;J\rangle]\), we obtain the following result.

\begin{theo}\label{theo-single}
If
$i \not\in I,~j\not\in J$, we have
$$ \CT_{\lambda_i,j} [M_k\langle I;J\rangle] =\sgn(p_{k+1}) \mathcal{Z}_{y_j} [M_k\langle I;J\rangle\leftarrow (i,j)],$$
where $p_{k+1}$ is the pivot item, i.e., the $(n+i,j)$th entry in $M_k\langle I;J\rangle$.
\end{theo}
\begin{proof}
Observe that for $j\not \in J$, the $j$th row is kept the same through the elementary operations on $M$ to obtain $M_k\langle I;J\rangle$.
Thus $y_j$ only appears in the $j$th denominator factor of $[M_k\langle I;J\rangle]$ so that
Proposition \ref{prop-singleCT} applies. In fact, this corresponds to the pivot operation on $M_k\langle I;J\rangle$, and the resulting rational function clearly has
matrix form $M_k\langle I;J\rangle\leftarrow (i,j)$ with the sign of $p_{k+1}$.
\end{proof}

\begin{lem}\label{lem-CT-cone}
Suppose $M$ is given in \eqref{equ-M}. Then for $k\leq r$, we have
$$\CT\limits_{\lambda_{i_k},j_k}\cdots\CT\limits_{\lambda_{i_2},j_2}\CT\limits_{\lambda_{i_1},j_1} [M]=\sgn\left(A\begin{pmatrix}
i_1,\dots,i_{k}\\
j_1,\dots,j_{k}
\end{pmatrix}\right)\mathcal{Z}_{\y}[M_k\langle \{i_1,\dots,i_{k}\};\{j_1,\dots,j_{k}\}\rangle],$$
provided that both sides are well-defined.
\end{lem}
\begin{proof}
The base case $k=1$ follows directly from Theorem \ref{theo-single}. Assume the equation holds for $k-1$. Thus we have $$\CT\limits_{\lambda_{i_{k-1}},j_{k-1}}\cdots\CT\limits_{\lambda_{i_1},j_1} [M]=\sgn\left(A\begin{pmatrix}
i_1,\dots,i_{k-1}\\
j_1,\dots,j_{k-1}
\end{pmatrix}\right)\mathcal{Z}_{\y}[M_{k-1}\langle \{i_1,\dots,i_{k-1}\};\{j_1,\dots,j_{k-1}\}\rangle].$$

Apply $\CT\limits_{\lambda_{i_k},j_k}$ to both sides. It suffices to prove that
\begin{equation*}
\CT\limits_{\lambda_{i_k},j_k}\mathcal{Z}_{\y}[M_{k-1}\langle\{i_1,\dots,i_{k-1}\};\{j_1,\dots,j_{k-1}\}\rangle]
=\sgn(p_k)\mathcal{Z}_{\y}[M_k\langle \{i_1,\dots,i_{k}\};\{j_1,\dots,j_{k}\}\rangle],
\end{equation*}
where $p_{k}$ is the pivot item, i.e., the $(n+i_k,j_k)$th entry in $M_{k-1}\langle\{i_1,\dots,i_{k-1}\};\{j_1,\dots,j_{k-1}\}\rangle$.

By Proposition \ref{prop-commu} and Theorem \ref{theo-single}, we obtain
\begin{align*}
\CT\limits_{\lambda_{i_k},j_k}\mathcal{Z}_{\y}[M_{k-1}\langle\{i_1,\dots,i_{k-1}\};\{j_1,\dots,j_{k-1}\}\rangle]
&=\mathcal{Z}_{\y}\CT\limits_{\lambda_{i_k},j_k}[M_{k-1}\langle\{i_1,\dots,i_{k-1}\};\{j_1,\dots,j_{k-1}\}\rangle]\\
&=\sgn(p_k)\mathcal{Z}_{\y}[M_k\langle \{i_1,\dots,i_{k}\};\{j_1,\dots,j_{k}\}\rangle].
\end{align*}
This completes the proof.
\end{proof}

The case $k=r$ of Lemma \ref{lem-CT-cone} is what we need.
\begin{theo}\label{theo-CT-cone}
Suppose $M$ is given in \eqref{equ-M}. Then we have if both sides are well-defined,
$$\CT\limits_{\lambda_{i_r},j_r}\cdots\CT\limits_{\lambda_{i_2},j_2}\CT\limits_{\lambda_{i_1},j_1}[M]=\sgn\left(A\begin{pmatrix}
i_1,\dots,i_{r}\\
j_1,\dots,j_{r}
\end{pmatrix}\right)\sigma_{K^v}(\y),$$
where $K^v= K^{\mydot}(M_r\langle\{j_1,j_2,\dots,j_{r}\}\rangle)$ is a shifted simplicial cone defined in \eqref{equ-shift}.
\end{theo}
\begin{proof}
By Lemma \ref{lem-CT-cone}, it suffices to prove that, as rational functions,
$$\mathcal{Z}_{\y}[M_r\langle\{j_1,j_2,\dots,j_{r}\}\rangle]=\sigma_{K^v}(\y).$$

For convenience, let us denote the matrix $M_r\langle\{j_1,j_2,\dots,j_{r}\}\rangle$ as $(\gamma_1,\dots,\gamma_{n-r+1})$.  By definition, the left-hand side is just $\sigma_{K'}(\y)$, where $K'=\{\gamma_{n-r+1}+\sum_{j=1}^{n-r}k_ j\gamma_j:~k_j\in\N\}$.
We complete the proof by showing $K'\cap\Z^n=K^v\cap\Z^n$.
On the one hand, $K'\cap\Z^n\subseteq K^v\cap\Z^n$ follows from the fact $K'\subseteq K^v$;
On the other hand, if $\gamma_{n-r+1}+\sum_{j=1}^{n-r}k_j\gamma_j\in\Z^n$ then $ k_j\in\Z$ for all $j$ by the submatrix $(\textit{id}_{n-r},\0)$ in $M_r\langle\{j_1,j_2,\dots,j_{r}\}\rangle$. This means $K'\cap\Z^n\supseteq K^v\cap\Z^n$.
\end{proof}

\begin{rem}
In Theorem \ref{theo-CT-cone}, if $\det\left(A\begin{pmatrix}
i_1,\dots,i_{r}\\
j_1,\dots,j_{r}
\end{pmatrix}\right)=\pm 1$, then by Equation \eqref{equ-Mk}, $M_r\langle\{j_1,j_2,\dots,j_{r}\}\rangle$ is already integral, so that its columns give the integral vertex and primitive generators. This result cannot be obtained directly from constant term extraction using Equation \eqref{equ-addifomula-A}.
\end{rem}

\begin{rem}
If \( A \) is unimodular, i.e., all \( r\times r\) minors belong to \(\{ 0,1,-1\} \), then Theorem \ref{theo-CT-cone} provides a new proof that the polyhedron \( P(A,\b) \) is integral.
\end{rem}

\section{The SimpCone algorithm according to a strategy}\label{SimpCone}
In this section, we develop the \texttt{SimpCone[S]} algorithm for solving the inhomogeneous system $A\alpha = \b$ using a strategy $\texttt{S}$, which is used to make choices when applicable.

\subsection{The Inhomogeneous System}
We need some notation to translate the constant term of Erat forms to that of matrix forms. Let $M$ be defined in \eqref{equ-M}.
Consider $M_k\langle I;J\rangle \in \Q^{(n+r) \times (n+1)}$ with rows indexed by $n+I$ and columns indexed by $J$ ignored. For $j\leq n$, the $j$th column $\gamma_j$ is said to be \emph{small} if its first nonzero entry is positive, and \emph{large} if otherwise.

Suppose we need to take the constant term in $\lambda_i$ for an $i \notin I$. For simplicity, write
$$[M_k\langle I;J\rangle] =F\langle J\rangle=\frac{\lambda_i^{b_0}u_{0}}{\prod_{j=1}^{n}(1-\lambda_i^{b_j}u_j)}.$$
The degree and pole order can be quickly computed using the formula
\begin{equation}\label{deg&pole}
\deg_{\lambda_i} F\langle J\rangle=b_0-\sum_{j=1}^{n}\chi(b_j>0)b_j\quad\text{and}\quad\ord_{\lambda_i} F\langle J\rangle=-b_0+\sum_{j=1}^{n}\chi(b_j<0)b_j,
\end{equation}
which is easily derived using the following facts:

\begin{align*}
\deg_{\lambda_i}\frac{1}{1-\lambda_i^{b_j}u_j}=
\begin{cases}
-b_j, & \mbox{if $b_j\geq 0$};\\
  0,  & \mbox{if $b_j<0$}
\end{cases}
\quad\text{and}\quad \ord_{\lambda_i}\frac{1}{1-\lambda_i^{b_j}u_j}=
\begin{cases}
 0  ,  & \mbox{if $b_j\geq 0$};\\
b_j,  & \mbox{if $b_j<0$}.
\end{cases}
\end{align*}

Fix a $j\not\in J$ and $j\neq n+1$. Denote by $p_{k+1}$ the pivot item, i.e.,
the $(n+i)$th entry in $\gamma_j$. Then $\gamma_j$ is said to be not contributing (with respect to $\lambda_i$) if $p_{k+1}=0$;
otherwise it is said to be \emph{contributing} if $p_{k+1} \gamma_j$ is small and \emph{dually contributing} if $p_{k+1}\gamma_j$ is large.
Temporarily set $\c_i$ and $\d_i$ to be the number of contributing and dually contributing columns respectively. Update by
$\c_i=\infty$ if $F\langle J\rangle$ is not proper in $\lambda_i$, and by $\d_i=\infty$ if $F\langle J\rangle$ has a nonnegative pole order. Observe that
$\c_i$ and $\d_i$ cannot be both $\infty$.

\begin{lem} \label{lem-Irule}
Let $M=\begin{pmatrix}
         \textit{id}_n&\0\\
         A &-\b
        \end{pmatrix}$ be defined in \eqref{equ-M}. Let $\c_i,~\d_i$ be computed for $[M_k\langle I;J\rangle]$.

        If $\c_i<\infty$, then we have the \emph{contributing formula}
   \begin{equation}\label{equ-contri-c}
     \CT_{\lambda_i}[M_k\langle I;J\rangle]=
                      \sum\limits_{\substack{j: \gamma_j
                   \text{ is contributing}}} \sgn(p_{k+1})\mathcal{Z}_{y_j} [M_k\langle I;J\rangle\leftarrow (i,j)];
   \end{equation}
 If $\d_i<\infty$, then we have the \emph{dual formula}
\begin{equation}\label{equ-dualcontri-d}
\CT_{\lambda_i}[M_k\langle I;J\rangle]= -\sum\limits_{j:\gamma_j
                   \text{ is dually contributing}}\sgn(p_{k+1})\mathcal{Z}_{y_j}[M_k\langle I;J\rangle\leftarrow (i,j)],
\end{equation}
where $p_{k+1}$ is the pivot item for $j$, i.e., the $(n+i,j)$th entry in $M_k\langle I;J\rangle$.
\end{lem}

\begin{proof}
If $\c_i<\infty$ then $[M_k\langle I;J\rangle]$ is proper in $\lambda_i$.
By Corollary \ref{cor-contri-dual}, we have
 \begin{equation*}
     \CT_{\lambda_i}[M_k\langle I;J\rangle]=
                      \sum\limits_{\substack{j: \gamma_j
                   \text{ is contributing}}}  \CT_{\lambda_i,j}[M_k\langle I;J\rangle],
   \end{equation*}
which is equivalent to Equation \eqref{equ-contri-c} by Theorem \ref{theo-single}.

If $\d_i<\infty$ then $[M_k\langle I;J\rangle]$ has a negative pole order in $\lambda_i$, i.e.,
$[M_k\langle I;J\rangle]$ vanishes when $\lambda_i\to 0$. Then Equation \eqref{equ-dualcontri-d} follows in a similar way.
\end{proof}
Now we can describe the algorithm as follows.
\begin{algor}[\texttt{SimpCone[S]}]\label{algor-Simpcone}
\ \\
Input: A matrix $(A,-\b) \in \Z^{r\times (n+1)}$ and a strategy $\texttt{S}$.\\
Output: A  simplicial cone decomposition of the polyhedron $P(A,\b)$ as in \eqref{equ-dec-K}.
\begin{enumerate}
  \item  Construct the set $\{(1, O_{0,1}, \varnothing, \varnothing )\}$ where $ O_{0,1}=M$ as in \eqref{equ-M}.
  \item For $i$ from $0$ to $r-1$, compute $O_{i+1}$ from $O_{i}$  as follows:
  \begin{enumerate}
    \item For each term $(s_{i,j}, O_{i,j}, I_{i,j}, J_{i,j})$ of $O_i$, choose a $\lambda_{k_0}$ to take constant term and select a formula in Lemma \ref{lem-Irule} according to $\texttt{S}$.
    \item Compute $s_{i,j}\CT_{\lambda_{k_0}} [O_{i,j}\langle I_{i,j}; J_{i,j}\rangle]$ with respect to $I_{i,j}$ and $J_{i,j}$ according to the choice, and collect the corresponding terms into $O_ {i+1}$.
  \end{enumerate}
  \item Output $O_r=\{(s_1, O_{r,1},J_{r,1}),\dots,(s_N, O_{r,N}, J_{r,N}\}\}$.
\end{enumerate}
\end{algor}

Now we describe several strategies.

\texttt{S0}: Natural strategy: Choose $k_0 = i+1$ and apply Lemma \ref{lem-Irule} with respect to $\lambda_{k_0}$. If both formulas apply, then select the contributing formula as given in \eqref{equ-contri-c} if $\c \leq \d$; otherwise, select the dual formula as given in \eqref{equ-dualcontri-d}. In other words, we successively eliminate $\lambda_1, \lambda_2, \dots$ and apply Lemma \ref{lem-Irule} accordingly.

\texttt{S1}: Choose $k_0 = r+1-i$. If the corresponding $b_0 \geq 0$, then use the dual formula; otherwise, use the contributing formula. This strategy, when applied to $A = (A', -\textit{id}_r)$, produces an algorithm identical with that in \cite{2017Polyhedral}, which deals with the model $A'\alpha' \geq \b$.

\texttt{S2}: Greedy strategy: For each $(s_{i,j}, O_{i,j}, I_{i,j}, J_{i,j})$ of $O_i$, we need to compute the constant term of $[O_{i,j} \langle I_{i,j}; J_{i,j} \rangle]$.
Compute $\c_k$ and $\d_k$ for each $k \notin I_{i,j}$ and find their minimum $m_0$. Let $k_0$ be the smallest row index such that $m_0 = \min(\c_{k_0}, \d_{k_0})$.
Then, select the contributing formula if $\c_{k_0} = m_0$; otherwise, select the dual formula. This strategy was used in \cite{2023Algebraic} with good performance.

\begin{proof}[Justification of Algorithm \ref{algor-Simpcone}]
By Corollary \ref{cor-contri-dual} and Proposition \ref{prop-commu}, we have
\begin{equation*}
 \CT_{\lambda_i}\mathcal{Z}_{\y} [M_k\langle I;J\rangle]=\mathcal{Z}_{\y}\CT_{\lambda_i} [M_k\langle I;J\rangle].
\end{equation*}

By Lemma \ref{lem-Irule} and the commutativity of $\CT_{\lambda_i}$ and $\mathcal{Z}_{\y}$, we ultimately obtain a weighted sum of $\mathcal{Z}_{\y}[M_r\langle \{j_1,j_2,\dots,j_{r}\}\rangle]$, which corresponds to a shifted simplicial cone by Theorem \ref{theo-CT-cone}.
This yields the desired decomposition as in \eqref{equ-dec-K}.
\end{proof}

\begin{rem}
In Lemma \ref{lem-Irule}, we can also give a formula when $\ord_{\lambda_i}[M_k\langle I;J\rangle]=0$. The additional term $[M_k\langle I;J\rangle]\big|_{\lambda_i=0}$ may include redundant denominator factors $(1)$. Consequently, the final outcome may involve lower dimensional cones.
\end{rem}


We provide an example to illustrate the algorithm.
\begin{exam}
Suppose $A=\begin{pmatrix}
         3 & 1 & -4 &-9 & -1&0\\
         2 & -1 & 1 &-3 &0&-1
       \end{pmatrix}$ and $\b=(1,-3)^t$. Compute the  simplicial cone decomposition of the polyhedron  $P=P(A,\b)$ using strategy \texttt{S2}.
\end{exam}

\begin{proof}[Solution]
In Step 1, we construct the set $O_0=\{(1, O_{0,1}, \varnothing, \varnothing )\}$ where $ O_{0,1}=\begin{pmatrix}
                           \textit{id}_6 &\0\\
                           A    & -\b
                         \end{pmatrix}$.
In Step 2, we observe that each column is small. Specifically, $\c_1=2, \d_1=3$ are the number of positive and negative entries, respectively in row $7$ excluding the last column. Similarly $\c_2=2, \d_2=3$. Note that we update $\c_2=\infty$ because $\deg_{\lambda_2}[O_{0,1}]=3-(2+1)=0$ according to Equation \eqref{deg&pole}. There is no further updates: for instance, $\ord_{\lambda_2}[O_{0,1}]=-3+(-1-3-1)=-8<0$. Therefore we choose $k_0=1$ and
eliminate $\lambda_1$ using the contributing formula as given in \eqref{equ-contri-c}. This gives
 $O_1=\{(1, O_{1,1}, \{1\}, \{1\}),(1, O_{1,2}, \{1\}, \{2\})\}$, where
      $$O_{1,1}=\begin{pmatrix}
{\color{red}1}&-1/3&4/3&3&1/3&0&1/3\\
{\color{red}0}&1&0&0&0&0&0\\
{\color{red}0}&0&1&0&0&0&0\\
{\color{red}0}&0&0&1&0&0&0\\
{\color{red}0}&0&0&0&1&0&0\\
{\color{red}0}&0&0&0&0&1&0\\
{\color{red}\textcircled{3}}&{\color{red}0}&{\color{red}0}&{\color{red}0}&{\color{red}0}&{\color{red}0}&{\color{red}0}\\
{\color{red}2}&-5/3&11/3&3&2/3&-1&11/3
\end{pmatrix}~\text{and}~O_{1,2}=\begin{pmatrix}
1&{\color{red}0}&0&0&0&0&0\\
-3&{\color{red}1}&4&9&1&0&1\\
0&{\color{red}0}&1&0&0&0&0\\
0&{\color{red}0}&0&1&0&0&0\\
0&{\color{red}0}&0&0&1&0&0\\
0&{\color{red}0}&0&0&0&1&0\\
{\color{red}0}&{\color{red}\textcircled{1}}&{\color{red}0}&{\color{red}0}&{\color{red}0}&{\color{red}0}&{\color{red}0}\\
5&{\color{red}-1}&-3&-12&-1&-1&2
\end{pmatrix}.$$

For $(1, O_{1,1}, \{1\}, \{1\})$, the second column is large and the corresponding entry $-\frac{5}{3}$ is negative, so it is contributing. By definition, $\c_2 = 4$ and $\d_2 = 1$, and there are no updates. Thus, we eliminate $\lambda_2$ using the dual formula as given in \eqref{equ-dualcontri-d}. This gives $\{(1, O_{2,1}, \{1,6\})\},$
where
$$ O_{2,1}=\begin{pmatrix}
-1/3&4/3&3&1/3&1/3\\
1&0&0&0&0\\
0&1&0&0&0\\
0&0&1&0&0\\
0&0&0&1&0\\
-5/3&11/3&3&2/3&11/3
\end{pmatrix}.$$
A similar analysis for $(1, O_{1,2}, \{1\}, \{2\})$ shows that we should use the contributing formula. This gives
$\{(1, O_{2,2}, \{1,2\})\},$ where
$$ O_{2,2}=\begin{pmatrix}
3/5&12/5&1/5&1/5&-2/5\\
11/5&9/5&2/5&-3/5&11/5\\
1&0&0&0&0\\
0&1&0&0&0\\
0&0&1&0&0\\
0&0&0&1&0
\end{pmatrix}.$$

In Step 3, we output $O_2=\{(1, O_{2,1}, \{1,6\}),(1, O_{2,2}, \{1,2\})\}.$

Note that if we select strategy \texttt{S1}, the output contains 5 cones.
\end{proof}

The \texttt{SimpCone[S]} algorithm can deal with matrix $A$ with parameter entries.
\begin{exam}\label{exam-a}
Suppose $M=\begin{pmatrix}
                           \textit{id}_4 &\0\\
                           A    & -\b
                         \end{pmatrix}$, where $A=\begin{pmatrix}
           1 & -2 & 1 & -1 \\
           w & 2 & 3 & -1
         \end{pmatrix}$ and  $\b=(1,-2)^t$.
Then, we can obtain the simplicial cone decomposition of the polyhedron $P=P(A,\b)$ in three cases as follows.
\begin{align}\label{cases}
\sigma_P(\y)=\begin{cases}
  \mathcal{Z}_{\y}[M_2\langle \{1,2\}\rangle]+\mathcal{Z}_{\y}[M_2\langle\{1,4\}\rangle], & \mbox{if $w<-1$}; \\
  \mathcal{Z}_{\y}[M_2\langle \{1,4\}\rangle], & \mbox{if $-1\leq w<1$}; \\
  0, & \mbox{if $w\geq 1$}.
\end{cases}
\end{align}
\end{exam}
\begin{proof}[Solution]
We shall delay making choices according to the parameter $w$, which is in column $1$, row $6$ (corresponding to $\lambda_2$).
Firstly, we find $\c_1 =2 $ and $\d_1 = 2$ without updates since $\deg_{\lambda_1}[M]<0$ and $\ord_{\lambda_1}[M]<0$.
The values of $\c_2$ and $\d_2$ depend on the parameter $w$. We eliminate $\lambda_1$ using the dual formula.
This minimizes the effect of the parameter $w$ and gives
$$\CT_{\lambda_1}[M]=\mathcal{Z}_{\y}[M_1\langle\{1\};\{2\}\rangle]+\mathcal{Z}_{\y}[M_1\langle\{1\};\{4\}\rangle],$$ where
\begin{small}$$M_1\langle\{1\};\{2\}\rangle=\begin{pmatrix}
 1&{\color{red}0}&0&0&0\\
 1/2&{\color{red}1}&1/2&-1/2&-1/2\\
 0&{\color{red}0}&1&0&0\\
 0&{\color{red}0}&0&1&0\\ {\color{red}0}&{\color{red}\textcircled{-2}}&{\color{red}0}&{\color{red}0}&{\color{red}0}
\\w+1&{\color{red}2}&4&-2&1\end{pmatrix}\    \ \text{and}\  \ M_1\langle\{1\};\{4\}\rangle=\begin{pmatrix}
1&0&0&{\color{red}0}&0\\ 0&1&0&{\color{red}0}&0
\\ 0&0&1&{\color{red}0}&0\\ 1&-2&1&{\color{red}1}&-1
\\ {\color{red}0}&{\color{red}0}&{\color{red}0}&{\color{red}\textcircled{-1}}&{\color{red}0}\\ w-1&4&2&{\color{red}-1}&3
\end{pmatrix}. $$\end{small}

Now, consider $\CT_{\lambda_2}[M_1\langle\{1\};\{2\}\rangle]$. We find $\c_2 = 2 + \chi(w > -1)$ and $\d_2 = \chi(w < -1)$.
Similarly, for $\CT_{\lambda_2}[M_1\langle\{1\};\{4\}\rangle]$, we have $\c_2 = 2 + \chi(w > 1)$ and $\d_2 = \chi(w < 1)$. Thus, we consider $w$ according to the three cases in \eqref{cases}.
We only provide details for the case $-1 \leq w < 1$. The other cases are similar.

When $-1<w<1$, by Equation \eqref{deg&pole}, we have $\deg_{\lambda_2}[M_1\langle\{1\};\{2\}\rangle]=-w-4<0$ and $\ord_{\lambda_2}[M_1\langle\{1\};\{2\}\rangle]=-3<0$, so that there is no update,
i.e.,  $\c_2=3, \d_2=0$.  We select the dual formula and obtain $\CT_{\lambda_2}[M_1\langle\{1\};\{2\}\rangle]=0$.
A similar analysis for $\CT_{\lambda_2}[M_1\langle\{1\};\{4\}\rangle]$ gives $\c_2=2, \d_2=1$. Applying the \textrm{dual formula} gives
$\CT_{\lambda_2}[M_1\langle\{1\};\{4\}\rangle]=\mathcal{Z}_{\y}[M_2\langle \{1,4\}\rangle].$
Adding the two equations above gives $\sigma_P(\y)=\mathcal{Z}_{\y}[M_2\langle \{1,4\}\rangle].$

Next, we explain why we include the boundary value $w=-1$ but exclude $w=1$. If \( w = -1 \), then \( \c_2 = 2 \) for \( \CT_{\lambda_2}[M_1\langle\{1\};\{2\}\rangle] \).  However,  the constant term is still $0$ by using the dual formula.
If \( w = 1 \), then \( \d_2 = 0 \) for \( \CT_{\lambda_2}[M_1\langle\{1\};\{4\}\rangle] \). Consequently, \( \CT_{\lambda_2}[M_1\langle\{1\};\{4\}\rangle] = 0 \), rather than \( \mathcal{Z}_{\mathbf{y}}[M_2\langle\{1,4\}\rangle] \).
\end{proof}
\subsection{The homogeneous System}
For the case $\b = \0$, we omit the rightmost column in $M$, since it is always the zero vector. This case simplifies significantly. Here, we assume $A$ has rank $r$ and $A\x = \0$ has a positive solution (in $\mathbb{P}^n$), or equivalently, the cone $K = P(A, \0)$ has dimension $n - r$. This assumption is natural, since otherwise, $\alpha_i$ is forced to be $0$ for some $i$, and the $i$th column of $A$ can be removed without losing information.

Stanley defined $A \leftarrow (i, j)$ by taking the $(i, j)$th entry of $A$ as the pivot, and similarly defined $A \leftarrow ((i_1, \dots, i_k); (j_1, \dots, j_k))$ (only for $(i_1, \dots, i_k) = (1, \dots, k)$). Clearly, $A \leftarrow ((i_1, \dots, i_k); (j_1, \dots, j_k))$ can be obtained from $M \leftarrow ((i_1, \dots, i_k); (j_1, \dots, j_k))$ by removing rows $1, 2, \dots, n, n + i_1, \dots, n + i_k$ and columns $j_1, \dots, j_k$.

\begin{theo}\label{theo-n&p}
Suppose the integral matrix $A=(a_{i,j})_{r\times n}$ has rank $r$ and the system $A\x=0$ has a positive solution.
Then each row of $A\leftarrow ((i_1,\dots, i_k);(j_1,\dots, j_k))$ contains at least one positive entry and at least one negative entry.
\end{theo}
\begin{proof}
Suppose $\alpha=(\alpha_1,\dots,\alpha_n)^t \in \mathbb{P}^n$ is a positive solution of $A\x=\0$.
We first prove the theorem for $k=0$. This follows easily
by: i) the row $s$ vector $R_s=(a_{s1},\dots, a_{sn})$ of $A$ cannot be the zero vector,
since otherwise $\rank(A)<r$; ii) the scalar product $R_s \cdot \alpha$ being $0$ implies that
$R_s$ must contain both negative and positive numbers at the same time.

Next, we claim that $B=A\leftarrow(i,j)$ has rank $r-1$ and also has a positive solution.
Then the theorem is followed by an iterative application of the claim and the $k=0$ case of the theorem.

To prove the claim, we assume $(i,j)=(1,1)$ without loss of generality by permuting rows and columns.
If we write $A=(A_1,A_2,\dots,A_n)$, then $B$ is obtained from $B'=(A_2-\frac{a_{12}}{a_{11}}A_1,\dots,A_n-\frac{a_{1n}}{a_{11}}A_1)$
by deleting the first row,  which consists of all zeroes. Thus, $B\x=\0\Leftrightarrow B'\x=\0$.

Now we show that $\alpha'=(\alpha_2,\dots,\alpha_n)^t$ is the desired positive solution for $B\x=\0$.
By $A\alpha=\0$, we have $\sum_{i=2}^{n}a_{1i}\alpha_i=-a_{11}\alpha_1$. Thus,
\begin{equation*}
  B'\alpha'=\sum_{i=2}^{n}(A_i-\frac{a_{1i}}{a_{11}}A_1)\alpha_i
  =\sum_{i=2}^{n}A_i\alpha_i-\frac{\sum_{i=2}^{n}a_{1i}\alpha_i}{a_{11}}A_1 =A \alpha =\0.
\end{equation*}
This completes the proof.
\end{proof}

\begin{rem}\label{rem-b=0}
The \texttt{SimpCone} algorithm in \cite{2023Algebraic} corresponds to the homogeneous case of the \texttt{SimpCone[S2]} algorithm, where
we require that $A\alpha=\0$ has a positive solution.
Its correctness is justified by Theorem \ref{theo-n&p}, which implies that the $\c_i$'s and $\d_i$'s are never updated
in the algorithm.
\end{rem}

For clarity, we state the following lemma to show our strategic flexibility.
\begin{lem}\label{lem-Algo-M}
Let $M=\begin{pmatrix}
         \textit{id}_n\\
         A
        \end{pmatrix}$ with $A$ as in Theorem \ref{theo-n&p}. For any
$i \not\in I$, we have
  \begin{equation}\label{equ-strategies}
\sum_{j:\substack{\gamma_j
                   \text{ is contributing} \\\text{or dually contributing}}} \sgn(p_{k+1})\mathcal{Z}_{\y} [M_k\langle I;J\rangle\leftarrow (i,j)]=0,
\end{equation}
where $p_{k+1}$ is the pivot item, i.e., the $(n+i,j)$th entry in $M_k\langle I;J\rangle$.
\end{lem}
\begin{proof}
By Theorem \ref{theo-n&p}, the $n+i$th row of $M_k\langle I;J\rangle$ has at least one positive entry and at least one negative entry. Applying Equation \eqref{deg&pole}, we find that both $\deg_{\lambda_i}([M_k\langle I;J\rangle])$ and $\ord_{\lambda_i}([M_k\langle I;J\rangle])$ are negative. Thus by Corollary \ref{cor-contri-dual}, we have $$\sum_{j:\substack{\gamma_j
                   \text{ is contributing} \\\text{or dually contributing}}} \CT_{\lambda, j}[M_k\langle I;J\rangle]=0.$$
This is just the desired \eqref{equ-strategies} by Theorem \ref{theo-single}.
\end{proof}

Lemma \ref{lem-Algo-M} allows us to establish certain linear relationships among the $\mathcal{Z}_{\y}[M_{r}\langle J\rangle]$'s.

\begin{exam}\label{exam-homo}
Suppose $A=\begin{pmatrix}
           1 & -2 & 1 & -1 \\
           -1 & 2 & 3 & -1
         \end{pmatrix}$ and $M=\begin{pmatrix}
                           \textit{id}_4 \\
                           A
                         \end{pmatrix}$.

Applying Algorithm \ref{algor-Simpcone} with strategy \texttt{S0} gives
\begin{equation*}
  \sigma_K(\y)=-\mathcal{Z}_{\y}[M_2\langle\{1,3\}\rangle]+\mathcal{Z}_{\y}[M_2\langle \{1,4\}\rangle]+\mathcal{Z}_{\y}[M_2\langle\{2,3\}\rangle]+\mathcal{Z}_{\y}[M_2\langle\{3,4\}\rangle],
\end{equation*}
and using strategy \texttt{S2} gives only the second term.
This equality can be explained
using Equation \eqref{equ-strategies}.
\end{exam}
\begin{proof}[Solution]
The matrix $A$ has 5 second-order nonzero minors, giving rise to $\sigma_{K_J}=\mathcal{Z}_{\y}[M_2\langle J\rangle]
$ where $J$ belongs to $\{\{1,3\},\{1,4\},\{2,3\},\{2,4\},\{3,4\}\}.$
We obtain the following linear relations among them:
\begin{align}\label{equ-linear}
\begin{cases}
 & \mathcal{Z}_{\y}[M_2\langle \{1,3\}\rangle]=\mathcal{Z}_{\y} [M_2\langle \{1,4\}\rangle];\\
 &\mathcal{Z}_{\y} [M_2\langle\{1,4\}\rangle]=\mathcal{Z}_{\y} [M_2\langle \{2,4\}\rangle]+\mathcal{Z}_{\y} [M_2\langle \{3,4\}\rangle];\\
 & \mathcal{Z}_{\y}[M_2\langle \{2,3\}\rangle]=\mathcal{Z}_{\y} [M_2\langle\{2,4\}\rangle],
 \end{cases}
\end{align}
which means $\sigma_{K_J}$ is spanned by two of them corresponding to $\sigma_{K_{\{2,4\}}}$ and $\sigma_{K_{\{3,4\}}}$.

Equation \eqref{equ-linear} is obtained by extending the application of \eqref{equ-strategies} to the following 8 matrices:
\begin{align*}
  &M_1\langle\{1\};\{1\}\rangle,M_1\langle\{1\};\{2\}\rangle,M_1\langle\{1\};\{3\}\rangle,M_1\langle\{1\};\{4\}\rangle, \\& M_1\langle\{2\};\{1\}\rangle,
  M_1\langle\{2\};\{2\}\rangle,M_1\langle\{2\};\{3\}\rangle,M_1\langle\{2\};\{4\}\rangle.
\end{align*}
We illustrate the idea by considering the fourth matrix. We have
 $$M_1\langle\{1\};\{4\}\rangle=\begin {pmatrix} 1&0&0& {\color{red}0}\\ 0&1&0&{\color{red}0}\\ 0&0&1&{\color{red}0}\\ 1&-2
&1&{\color{red}1}\\ {\color{red}0}&{\color{red}0}&{\color{red}0}&{\color{red}\textcircled{-1}}\\ -2&4&2&{\color{red}-1}
\end{pmatrix}.$$
By Equation \eqref{equ-strategies}, we derive the relationship
$$-\mathcal{Z}_{\y} [M_2\langle \{1,4\}\rangle]+\mathcal{Z}_{\y} [M_2\langle \{2,4\}\rangle]+\mathcal{Z}_{\y} [M_2\langle \{3,4\}\rangle]=0.$$
In this way, we obtain $8$ linear equations, which are linearly equivalent to those in \eqref{equ-linear}.

These relations can be verified by direct computation. We have
$$M_2\langle \{1,4\}\rangle=\begin{pmatrix}
                                  2 &1 \\
                                  1 &0 \\
                                  0 &1 \\
                                  0 &2
                    \end{pmatrix},~M_2\langle\{2,4\}\rangle=\begin{pmatrix} 1&0\\ {1/2}&{
-1/2}\\ 0&1\\ 0
&2\end{pmatrix},~
M_2\langle \{3,4\}\rangle=\begin{pmatrix}1&0\\ 0&1
\\ 1&-2\\ 2&-4\end{pmatrix}
.$$
Consequently, we have
\begin{align*}
\mathcal{Z}_{\y}[M_2\langle \{2,4\}\rangle]+\mathcal{Z}_{\y}[M_2\langle \{3,4\}\rangle]
&=\frac{1+y_1y_3y_4^2}{(1-y_1^2y_2)(1-y_2^{-1}y_3^2y_4^4)}+\frac{1}{(1-y_1y_3y_4^2)(1-y_2y_3^{-2}y_4^{-4})}\\
&=\frac{1}{(1-y_1^2y_2)(1-y_1y_3y_4^2)}=\mathcal{Z}_{\y}[M_2\langle \{1,4\}\rangle].
\end{align*}
This verifies the second equality in \eqref{equ-linear}. The others can be verified similarly.
\end{proof}

Theorem \ref{theo-n&p} provides a positive solution of \cite[Remark 5.4]{2007A}. This, together with Algorithm \ref{algor-Simpcone}, leads to a new proof of
Stanley's reciprocity theorem.
Before introducing Stanley's reciprocity theorem, for a cone \( K = P(A,\0) \), we define \(\bar{\sigma}_K(\mathbf{y}) := \sum_{\alpha \in K \cap \mathbb{P}^n} \mathbf{y}^\alpha \).
\begin{cor}[Stanley's reciprocity theorem]
\label{cor-3-recip}
Let $A$ be an $r$ by $n$ integral matrix of full rank $r$. If
there is at least one $\alpha\in \mathbb{P}^n$ such that $A\alpha
=\0$, then we have as rational functions,
\begin{equation*}
    \sigma_K(\y)= (-1)^{n-r} \bar{\sigma}_K(\y^{-1}).
\end{equation*}
\end{cor}
\begin{proof}
 First, consider $\sigma_K(\y)=\CT_{\Lambda} [M]$, where $M=\left[\begin{pmatrix}
         \textit{id}_n\\
         A
        \end{pmatrix}\right]$, by working in the field $G$ of iterated Laurent series.
For any strategy \texttt{S}, applying Algorithm \ref{algor-Simpcone} gives
$$ \CT_{\Lambda} [M] = \sum \CT\limits_{\lambda_{i_r},j_r}\cdots\CT\limits_{\lambda_{i_2},j_2}\CT\limits_{\lambda_{i_1},j_1}[M],$$
where the sum ranges over a finite set $\mathcal{F}$ of sequences $((i_1,j_1),\dots, (i_r,j_r))$.

Furthermore, the constant term formula for $\bar{\sigma}_K(\y)$ is known to be
\begin{align*}
\bar{\sigma}_K(\y)&=\CT_{\Lambda}\prod_{j=1}^n \frac{ \Lambda^{A_j} y_j}{ 1-\Lambda^{A_j} y_j}=(-1)^n\CT_{\Lambda} \frac{1}{\prod_{j=1}^n (1-\Lambda^{-A_j} y_j^{-1})}\\
&\Longrightarrow(-1)^n\CT_{\Lambda}\left[-M\right] \big|_{y_\ell\to y_\ell^{-1}}=(-1)^n\CT_{\Lambda}\left[M'\right]=\bar{\sigma}_K(\y^{-1}),
\end{align*}
where $M'=\begin{pmatrix}
         \textit{id}_n\\
         -A
        \end{pmatrix}$ and the working field becomes $\bar{G}$. The rule is that $y^\alpha <1$ in $G$ if and only if
$y^\alpha>1$ in $\bar G$.

 Thus the contributing property remains the \emph{same} in $\bar{G}$: If the $j$th column of $M_k\langle I;J\rangle$ is $\gamma=(\gamma_1,\gamma_2)^t$, where $\gamma_1\in\Q^{n}$ and $\gamma_2\in\Q^{r}$, then the $j$th column of $M'_k\langle I;J\rangle$ is $\gamma'=(\gamma_1,-\gamma_2)^t$. Clearly,
 $\gamma ~\text{is~small in } G \Leftrightarrow -\gamma' ~\text{is~small in } \bar{G}$.

Applying Algorithm \ref{algor-Simpcone} to $[M']$ according to the same strategy \texttt{S}. We obtain
$$\CT_{\Lambda}\left[M' \right]=  \sum  \CT\limits_{\lambda_{i_r},j_r}\cdots\CT\limits_{\lambda_{i_2},j_2}\CT\limits_{\lambda_{i_1},j_1} [M'],$$
where the sum ranges over the same set $\mathcal{F}$.

The result is then followed by
\begin{align*}
\CT\limits_{\lambda_{i_r},j_r}\cdots\CT\limits_{\lambda_{i_2},j_2}\CT\limits_{\lambda_{i_1},j_1} [M]
=&\sgn\left(A\begin{pmatrix}
i_1,\dots,i_{r}\\
j_1,\dots,j_{r}
\end{pmatrix}\right)\mathcal{Z}_{\y}[M_r\langle \{i_1,\dots,i_{r}\};\{j_1,\dots,j_{r}\}\rangle]\\
=&(-1)^r\sgn\left((-A)\begin{pmatrix}
i_1,\dots,i_{r}\\
j_1,\dots,j_{r}
\end{pmatrix}\right)\mathcal{Z}_{\y}[M'_r\langle \{i_1,\dots,i_{r}\};\{j_1,\dots,j_{r}\}\rangle]\\
=&(-1)^r\CT\limits_{\lambda_{i_r},j_r}\cdots\CT\limits_{\lambda_{i_2},j_2}\CT\limits_{\lambda_{i_1},j_1}  [M'].
\end{align*}
\end{proof}

\begin{rem}
The proof naturally extends for Stanley's reciprocity domain theorem, and a generalization \cite[Corollary 3.7]{2007A}.
\end{rem}

\section{On Unimodular Cone Decomposition}\label{sec:unimodular}
After obtaining the simplicial cone decomposition as in \eqref{equ-dec-K} in Step 1), one needs Steps 2) and 3) to solve the lattice point counting problem.
For Step 3), we refer to \cite{XinTodd}. Here, we discuss Step 2) on unimodular cone decomposition of a given simplicial cone $K$.

\subsection{Existing algorithms}
A classical approach for Step 2) is to use Barvinok's key result.
\begin{theo}[\cite{Ba1999}, Theorem 4.2]\label{theo-Ba1999}
Fix the dimension $d$. There exists a polynomial-time algorithm that, given a rational polyhedral cone $K\subset \R^d$, computes unimodular cones $K_i : i\in I$ and numbers $s_i\in\{-1,1\}$ such that
\begin{equation}\label{equ-unimodular}
\sigma_K(\y)=\sum_{i\in I}s_i \sigma_{K_i}(\y)=\sum_{i \in I} s_i\frac{\y^{\gamma_{i,0}}}{\prod_{k=1}^{n-1} (1- \y^{\gamma_{i,k}})}.
\end{equation}
\end{theo}
Barvinok's algorithm shows that when the dimension is fixed, the size $|I|$ can be bounded by a polynomial in $\ind(K)$. Such a sum is usually referred to as a \emph{short sum}. Barvinok's idea has been implemented by the C-package \texttt{LattE}, which is appraised as the state of the art algorithm. See \cite{LattE} and a manual of \texttt{LattE}  in \cite{LattEManual}.

There is a parallel development for unimodular cone decomposition for a special type of simplicial cone, named knapsack cone
or denumerant cone, as discussed in \cite{xin2024combinatorial}. In constant term language, the \texttt{DecDenu} algorithm provides a rapid computation of the underlined constant term
\begin{equation}\label{equ-single-DecDenu}
\CT_\lambda \frac{\lambda^{b}}{\underline{(1- \lambda^{a_1} y_1)} \prod_{i=2}^n (1- \lambda^{a_i} y_i)}
\end{equation}
as the sum in \eqref{equ-unimodular}. Note that when $a_1=1$, the constant term is a single term obtained by a simple substitution.

Computer experiments indicate that the \texttt{DecDenu} algorithm may outperform the \texttt{LattE} package for random $a_i \in \mathbb{N}$, despite the fact that the \texttt{DecDenu} algorithm has not been proven to be polynomial-time.

Now we describe how to use the \texttt{DecDenu} algorithm to effectively decompose $\sigma_P(\y)$ as a \emph{nearly short sum}, which is the right-hand side of \eqref{equ-unimodular}.
We add ``nearly" here because we cannot prove a polynomial bound.

\subsection{On simplicial cone of the form $K^{\mydot}(M_r\langle J\rangle)$}
We consider $K^v=K^{\mydot}(M_r\langle J\rangle)$. We assume $J=[r]$ without loss of generality.
In \eqref{equ-Mk}, we have  obtained
\begin{equation}\label{equ-Mr}
M_r\langle [r]\rangle=\begin{pmatrix}
-{A_1}^{-1}A_2&{A_1}^{-1}\b\\
\textit{id}_{n-r}&\0
\end{pmatrix}.
\end{equation}

When $r=1$, $A_1=(a_1)$, the cone is referred to as a knapsack cone or denumerant cone in \cite{xin2024combinatorial}.

We now outline the procedure for decomposing $K^v$ by iterative use of the \texttt{DecDenu} algorithm. The complexity of this procedure is measured by
$\det(A_1)$. This differs from that in \texttt{LLLCTEuclid}, which also iteratively uses the \texttt{DecDenu} algorithm, but the complexity has no obvious bound.

We first compute the Smith normal form of the submatrix $A_1$ as $H = UA_1V$, where $U$ and $V$ are unimodular matrices and $H$ is a diagonal matrix with diagonal entries $h_{1},h_{2},\dots,h_{r}$
satisfying that $h_{i}|h_{i+1}$ for all $1\leq i\leq r-1$.

Now we construct the matrix
\begin{align}\label{eq_N}
 \hat{M}=\begin{pmatrix}
         \textit{id}_{r} & \0 & \0\\
         \0 & \textit{id}_{n-r} & \0 \\
         H & U A_2  &-U\b
      \end{pmatrix}
\end{align}
such that
\begin{align}\label{e-M-M^}
    \hat{M}_r\langle [r]\rangle = \begin{pmatrix}
-V^{-1}{A_1}^{-1}A_2&V^{-1}{A_1}^{-1}\b\\
\textit{id}_{n-r}&\0
\end{pmatrix} = \text{diag}(V^{-1},\textit{id}_{n-r})  M_r\langle [r]\rangle.
\end{align}
This induces an isomorphism from the shifted cone $K^{\mydot}(M_r\langle [r]\rangle)$ to the shifted cone $K^{\mydot}(\hat{M}_r\langle [r]\rangle)$
through the unimodular transformation $\text{diag}(V,\textit{id}_{n-r})$.

Define the action of a nonsingular matrix \( W = (w_{ij})_{n \times n} \) by
$$ W \circ F(\y) = F(\y^{W e_1}, \y^{W e_2}, \dots, \y^{W e_n}), \quad \text{where} \quad  \y^{W e_i} = y_1^{w_{1i}} \cdots y_n^{w_{ni}}. $$

Then we have
\begin{align}\label{equ-M-M^}
  \mathcal{Z}_{\y}[ M_r\langle [r]\rangle] & = \mathcal{Z}_{\y}[ \text{diag}(V,id_{n-r}) \hat M_r\langle [r]\rangle]  \\
 \notag  & =\text{diag}(V,\textit{id}_{n-r})\circ \mathcal{Z}_{\y}[\hat M_r\langle [r]\rangle] \\
 \notag  & =\text{diag}(V,\textit{id}_{n-r})\circ \CT_{\lambda_{r},r}\cdots\CT_{\lambda_{2},2}\CT_{\lambda_{1},1} [\hat{M}].
\end{align}

For this iterated constant term, observe that
\begin{align*}
[\hat{M}] &= \frac{\Lambda^{U\b}}{(1-\lambda_1^{h_{1}} y_1)(1- \lambda_2^{h_{2}} y_2)\cdots (1- \lambda_r^{h_{r}} y_r) \prod_{i=1}^{n-r} (1- \Lambda^{\alpha_{r+i}}y_{r+i}) },
\end{align*}
where $\alpha_{r+i}$ is the $i$-th column of $UA_2$. Then the iterated constant term
\begin{equation}\label{e-success-CT}
\CT\limits_{\lambda_{r},r}\cdots\CT\limits_{\lambda_{2},2}\CT\limits_{\lambda_{1},1} [\hat{M}]=\CT\limits_{\lambda_{r}}\frac{1}{\underline{1-\lambda_r^{h_{r}} y_r}}\cdots\CT\limits_{\lambda_{1}} \frac{1}{\underline{1-\lambda_1^{h_{1}}y_1}}\cdot\frac{\Lambda^{U\b}}{\prod_{i=1}^{n-r} (1- \Lambda^{\alpha_{r+i}}y_{r+i})}
\end{equation}
can be successively computed using the \texttt{DecDenu} algorithm.

We summarize the above as the following algorithm:
\begin{algor}\label{alg_simpconeANDdecdenu}
  Input: Integral matrix $A$ and $\b$ defining the polyhedron $P=P(A,\b)$.

  Output: The nearly short sum of rational functions for $\sigma_P(\y)$ as in \eqref{equ-unimodular}.

  \begin{itemize}
    \item[1] Apply \texttt{Simpcone[S]} to obtain the set $O_r=\{(s_1, O_{r,1},J_{r,1}),\dots,(s_N, O_{r,N}, J_{r,N})\}$.

    \item[2] For each pair $(s_i, J_{r,i})$ (we omit the $O_{r,i}$) in $O_r$, compute $s_i\mathcal{Z}_\y [M_r\langle J_{r,i}\rangle]$ as follows.
    \begin{itemize}
      \item[2.1] Construct the matrix $\hat{M}$ corresponding to $J_{r,i}$  by \eqref{eq_N}.
      \item[2.2] Compute $\mathcal{Z}_\y[\hat{M}_r\langle J_{r,i}\rangle]$ using \eqref{e-success-CT} by applying the \texttt{DecDenu} algorithm.
      \item[2.3] Apply $s_i\text{diag}(V,\textit{id}_{n-r})$ to each rational function from the above.
    \end{itemize}
     \item[3] Take the sum over all $i$ gives the desired output.
  \end{itemize}
\end{algor}

\begin{remark}\label{rem-probability}
If $h_{i}=1$, then the operation $\CT_{\lambda_i,i}$ acts simply through a substitution. Thus in Step 2.2,
the action of $\CT_{\lambda_i,i}$ is only critical for those $h_i>1$. According to \cite[Theorem 4.13]{2015The},
if $A_1 \in \mathbb{Z}^{r\times r}$ is a random nonsingular matrix. Its Smith Normal Form $H$ has $84 \%$ probability
of having $h_{i}=1$ for all $i<r$, and $99 \%$ probability of having $h_{i}=1$ for $i<r-1$.
This indicates that in Step 2.2, we rarely use the \texttt{DecDenu} algorithm more than two times.
\end{remark}
Conclusively, based on the preceding analysis, we arrive at the following result for the shifted simplicial cone $K^v$.
\begin{theo}\label{theo-Denum}
Suppose a shifted simplicial cone $K^v$ is defined as  $K^{\mydot}(WM_r\langle J\rangle)$, where $W$ is a unimodular matrix. Then, utilizing Step 2 of Algorithm \ref{alg_simpconeANDdecdenu} may unimodularly decompose $\sigma_{K^v}(\y)$ as a nearly short sum in \eqref{equ-unimodular}.
\end{theo}
\begin{proof}
  Note that $K^v$ is isomorphic to the shifted cone $K^{\mydot}(M_r\langle J\rangle)$ through an unimodular transformation and
 $\sigma_{K^{v}}(\y)= W \circ  \mathcal{Z}_\y [M_r\langle J\rangle]$. By Step 2 of Algorithm \ref{alg_simpconeANDdecdenu},
 $\mathcal{Z}_\y [M_r\langle J\rangle]$ can be computed.
\end{proof}

\subsection{On general full-dimensional simplicial cone}
Any simplicial cone $K$ can be transformed into a full-dimensional simplicial cone so that Barvinok's algorithm applies to
give the desired decomposition.

Assume $K$ is full-dimensional. We present an alternative approach in our framework.
We construct a homogeneous linear Diophantine system $A\alpha=\0$ and decompose $\sigma_K(\y)$ using Theorem \ref{theo-Denum}.

More precisely, we use the encoding $K^*=K(B)$, where $B\in\Z^{d\times d}$.
Then $K=K((B^{-1})^{t})$. We construct $A=(A_1,A_2)=(B^t,-\textit{id}_{d})$. By \eqref{equ-Mr}, we have
$$M_d\langle [d]\rangle=\begin{pmatrix}
(B^{-1})^{t}\\
\textit{id}_{d}
\end{pmatrix}.$$
There exists an isomorphism from the cone for $(B^{-1})^{t}$ to the cone for $M_d\langle [d]\rangle$
through a unimodular transformation as follows.

\begin{equation}\label{equ-K-K*}
\begin{pmatrix}
(B^{-1})^{t}\\
\0
  \end{pmatrix}=\begin{pmatrix}
          \textit{id}_{d} & \0 \\
         -B^t & \textit{id}_{d}
       \end{pmatrix} M_d\langle [d]\rangle.
\end{equation}

Then Theorem \ref{theo-Denum} can be applied to express $\sigma_K(\y)$ in the form given by Equation \eqref{equ-unimodular}.
Moreover, we can make further observations here.

Let the Smith normal form of the $B^{t}$ be given by $H = UB^{t}V$.  By Equations \eqref{equ-K-K*} and \eqref{e-M-M^}, we have
\begin{equation*}
\begin{pmatrix}
(B^{-1})^{t}\\
\0
  \end{pmatrix}=W\cdot \hat{M}_d\langle [d]\rangle=\begin{pmatrix}
          \textit{id}_{d} & \0 \\
         -B^t & \textit{id}_{d}
       \end{pmatrix} \begin{pmatrix}
          V & \0 \\
         \0 & \textit{id}_{d}
       \end{pmatrix}\hat{M}_d\langle [d]\rangle.
\end{equation*}

Applying Equation \eqref{e-success-CT} yields
\begin{equation*}
\sigma_K(\y)=W\circ\mathcal{Z}_{\y}[\hat{M}_d\langle [d]\rangle]=W\circ\CT\limits_{\lambda_{d}}\frac{1}{\underline{1-\lambda_d^{h_{d}} y_d}}\cdots\CT\limits_{\lambda_{1}} \frac{1}{\underline{1-\lambda_1^{h_{1}}y_1}}\cdot\frac{1}{\prod_{i=1}^{d} (1- \Lambda^{\alpha_{d+i}}y_{d+i})}.
\end{equation*}
If $h_i=1$ for all $i<d$, then the above constant term simplifies to a constant term as in \eqref{equ-single-DecDenu}.
Thus, Remark \ref{rem-probability} says that a general full-dimensional cone $K$, under a unimodular transformation, is very likely to be a denumerant cone.

For general $h_i$, repeated application of Equation \eqref{equ-addifomula-A} gives
\begin{equation}\label{equ-generating-K-unit}
\sigma_K(\y)=W\circ\left(\frac{1}{\det(H)}\sum_{j_1=1}^{h_1}\cdots \sum_{j_d=1}^{h_d}\frac{1}{\prod_{i=1}^{d}(1-y_{d+i}\hat{\y}^{-H^{-1}\cdot \alpha_{d+i}}(\zeta_{1,j_1},\dots,\zeta_{d,j_d})^{\alpha_{d+i}})}\right),
\end{equation}
where $\hat{\y}=(y_1,y_2,\dots,y_d)$, and $\zeta_{t,j_t}$ ranges over all $h_t$th roots of unity.
Thus, we have written $\sigma_K(\y)$ as a sum of $\det(B)= h_1h_2\cdots h_d$ terms using roots of unity.

Equation \eqref{equ-generating-K-unit} is superior to the following result obtained in \cite{diaz}, as can be easily seen from concrete examples.
\begin{theo}[Diaz-Robins]\label{t-Diza-Robins}
Suppose the cone $ K=K(C) \subset \R^{d} $, where $C=(c_{ij})=(C_{1} , \dots , C_{d})\in\Z^{d\times d}$ is lower-triangular. Let
$ q_{k} := c_{11} \cdots c_{kk} \ ( k = 1, \dots , d ) $ and
$\mathcal{ G }:= \left( \Z / q_{1} \Z \right) \times \dots \times \left( \Z / q_{d} \Z \right) $. Then
  \[ \sum_{ \alpha \in K \cap \Z^{d} } e^{ - 2 \pi  \alpha^t \cdot \cs  }  = \frac{ 1 }{ 2^{d} |\mathcal{G}| } \sum_{ \mathbf{r} \in \mathcal{G} } \prod_{ k=1 }^{ d } \left( 1 + \coth \frac{ \pi }{ q_{k} } \left( (\cs + i \mathbf{r})^t\cdot C_{k} \right) \right) \ . \]
\end{theo}
%
\subsection{Examples}
The effectiveness of the \texttt{SimpCone} algorithm has been demonstrated in \cite[Section 7]{2023Algebraic}, where the algorithm was employed for volume computation.
In particular, it resolved the open problem for the volume of order $7$ magic square polytope. Here,
we present two examples of cone decompositions to illustrate the effectiveness of our approach to unimodular cone decompositions.

\noindent \emph{Magic Square Cone of Order $n$}, denoted \( MSC_n \), is characterized by the following system of linear constraints:
\begin{align*}
a_{i,1} + a_{i,2} + \cdots + a_{i,n}-m &=0, & \text{for } 1 \leq i \leq n, \\
a_{1,j} + a_{2,j} + \cdots + a_{n,j}-m &=0, & \text{for } 1 \leq j \leq n, \\
a_{1,1} + a_{2,2} + \cdots + a_{n,n}-m &=0, & a_{n,1} + a_{n-1,2} + \cdots + a_{1,n}-m &=0.
\end{align*}
This cone is closely related to the well-known challenges in counting magic squares.
For the case of \( n = 5 \), the decomposition of \( MSC_5 \) was first addressed in \cite{J2004Short}, and the algorithm was implemented within the \texttt{LattE} package. Initially, the triangulation of \( MSC_5 \) yields \( 6705 \) simplicial cones. Subsequently, the application of the Barvinok algorithm results in \( 24916 \) unimodular cones.

Utilizing Algorithm \ref{alg_simpconeANDdecdenu} does not yield a shorter sum, but not far away.
Indeed, the \texttt{SimpCone} method in Step 1 produces \( 8869 \) simplicial cones. This leads to a final sum of \( 61967 \) terms.
However, if we use Equation \eqref{equ-generating-K-unit}, then we obtain a sum of 15033 terms,  with each simplicial cone corresponding to $\leq 12$ terms.

\medskip
\noindent \emph{Random examples.}
We generate random matrices $A = (a_{i,j})_{4\times n}$ of rank $4$ with $|a_{i,j}| \leq 100$ and compute the generating functions of the corresponding cone $K=\{\alpha\in\R_{\geq 0}^n: \ A\alpha=\0\}$. For each cone $K$, we compare the performances using Algorithm \ref{alg_simpconeANDdecdenu} and \texttt{LattE} package. The triangulation in \texttt{LattE} is performed on $K^*$ in the dual space and then dualized back. This is because such a cone $K$ has far fewer facets (which are known to correspond to extreme rays of $K^*$) than extreme rays.

In Table \ref{table_randcase}, we present the data obtained from the two methods. Specifically, we record the count of simplicial cones generated by the \texttt{SimpCone} method, as well as the aggregate number of unimodular cones. Analogously, for the \texttt{LattE} package, we document the simplicial cones derived from triangulation, along with the corresponding count of unimodular cones.
\begin{table}[htbp]
\begin{tabular}{c|c|cc|cc|}
\hline
& & \multicolumn{2}{c|}{Alg. \ref{alg_simpconeANDdecdenu}} & \multicolumn{2}{c|}{LattE package}     \\ \hline
&size of A     & \multicolumn{1}{l|}{Simplicial cone}        & unimodular            & \multicolumn{1}{l|}{Simplicial cone} & unimodular   \\ \hline
Ex.1& 4 $\times$ 13 & \multicolumn{1}{l|}{3}   & 10317491  & \multicolumn{1}{l|}{27}  & $>$400 million($>$15 h) \\ \hline
Ex.2 & 4 $\times$ 13 & \multicolumn{1}{l|}{9}   & 170696890 & \multicolumn{1}{l|}{21}  & $>$400 million($>$15h) \\ \hline
Ex.3& 4 $\times$ 10 & \multicolumn{1}{l|}{22}  & 4655003   & \multicolumn{1}{l|}{19}  & 4261409      \\ \hline
Ex.4&4 $\times$ 10 & \multicolumn{1}{l|}{8}   & 981565    & \multicolumn{1}{l|}{18}  & 3471124      \\ \hline
Ex.5&4 $\times$ 8  &  \multicolumn{1}{l|}3    &4812       &\multicolumn{1}{l|} 4     & 7434         \\ \hline
Ex.6&4 $\times$ 8  &  \multicolumn{1}{l|}8    &11296      &\multicolumn{1}{l|} 6     & 10313        \\ \hline
\end{tabular} \caption{Data for  Alg. \ref{alg_simpconeANDdecdenu} and \texttt{LattE}. The ``$>$" symbol means to interrupt the calculation after 15 hours.
  \label{table_randcase}}
\end{table}

From the data on Examples $1$ and $2$, it is natural to ask: Does
the \texttt{SimpCone} algorithm produces fewer simplicial cones than triangulation methods, at least for a certain class of cones?
We cannot answer this question because there are too many ways to triangulate a cone in geometry. We conclude that \texttt{LattE}'s triangulation
the method can be optimized.

Consider Example $1$, where
$$A=\left(\begin{array}{ccccccccccccc}
11 & -26 & -36 & -3 & 61 & 28 & -89 & -7 & -60 & 4 & -94 & 19 & -35
\\
 10 & 1 & -27 & -32 & -60 & -35 & -23 & -63 & 77 & 10 & 43 & -12 & 88
\\
 60 & -16 & -17 & 86 & -24 & 51 & 21 & 37 & -65 & -54 & -60 & -68 & -72
\\
 79 & -33 & 9 & 17 & 37 & 51 & -9 & 80 & -84 & 52 & 24 & -13 & -7
\end{array}\right).$$
The corresponding cone $K$ has dimension $9$ and 25 extreme rays. The \texttt{SimpCone} algorithm outputs $3$ simplicial cones
with the help of two additional fake extreme rays. This is the best possible decomposition as each simplicial cone has $9$ rays
and all $25$ extreme rays must be included. Indeed, \texttt{LattE} produces  27 simplicial cones.

We have used Maple to re-implement the triangulation method from \cite{verdoolaege2005computation}, a method adopted in
 \texttt{LattE}. We find that, in the sense of isomorphism, certain ``random height" choices lead to the output of the \texttt{SimpCone} algorithm.
We have difficulty understanding why the \texttt{LattE} package consistently gives the same output in our repeated computations.

\section{Future Research Projects}\label{sec:project}
We have developed the \(\texttt{SimpCone[S]}\) algorithm for simplicial cone decomposition within the framework of algebraic combinatorics. Our decomposition differs from the triangulation method in geometry, as discussed in Section 2.1.

The \(\texttt{SimpCone[S]}\) algorithm includes the \(\texttt{Polyhedral Omega}\) algorithm as a special case. Upon examining the details of the \(\texttt{Polyhedral Omega}\) algorithm, one can observe that each elimination of \(\lambda_i\) can be explained by Brion's theorem. One might ask whether the \(\texttt{SimpCone}\) algorithm (in the homogeneous case) can be interpreted geometrically or obtained through some form of triangulation. The answer is likely negative: For a given cone \( K \), a simplicial cone decomposition of the form \(\sigma_K(\y) = \sum \pm \sigma_{K_i}(\y)\) does not necessarily imply that \(\sigma_K^v(\y) = \sum \pm \sigma_{K_i^v}(\y)\), which is a property that holds for triangulations.

The \(\texttt{SimpCone[S]}\) algorithm has given rise to the following four research projects.

\medskip Project 1 involves enhancing the algorithm by leveraging its flexibility. In addition to varying the choice of \(\texttt{ S} \), we can also change the working field \( G \) by permuting the order of the \( y \)-variables, which does not affect the series expansion in \eqref{equ-inhomo}. Furthermore, we can utilize the fact that \( A\alpha = \mathbf{0} \Leftrightarrow W A\alpha = \mathbf{0} \) for any nonsingular matrix \( W \). Therefore, we can select a matrix \( W \) such that the first row of \( WA \) has a minimum number of positive entries, which is our \( c_1 \). Consequently, the elimination of \( \lambda_1 \) will produce \( c_1 \) terms. This approach is applicable to each \( O_{i,j}\langle I_{i,j};J_{i,j}\rangle \). This leads to the following problem:
\begin{prob}
Given a matrix \( A \), find a row vector \( w \) such that \( wA \) has the minimum number of positive entries among all choices of nonzero \( w \).
\end{prob}

The problem is trivial when \( r = 1 \), since \( w \) can be either \( 1 \) or \( -1 \). If we have a fast method for solving the above problem, then we can employ a locally optimal strategy to minimize the number of outputs, at least when \( r \) is small and \( n \) is large.

\medskip Project 2 focuses on the computation of \(\sigma_K(\y)\) for a full-dimensional simplicial cone \(K\). A known fact is that the numerator of \(\sigma_K(\y)\) is a sum of \(\ind(K)\) monomials, and hence can be enumerated when \(K\) has a small index. The work in \cite{K2007A} introduces new algorithms based on irrational signed decomposition in the primal space and the known fact. However, it has been reported that \(\ind(K)\) might be large even if \(\ind(K^*)\) is small. Consequently, if we work in the dual space, \(K^*\) must be decomposed into unimodular cones to dual back.

Nevertheless, Equation \eqref{equ-generating-K-unit} provides an alternative solution when \(\ind(K^*)\) is small. Indeed, a fast algorithm was presented in \cite{2023An} for the case where \(K\) is a denumerant cone. Our second project aims to extend these ideas for general \(K\).

\medskip Project 3 involves extending the analysis to parametric polyhedra \( P = P(A; \mathbf{b}) \). The case where \(\mathbf{b}\) is parametric has been considered in \cite{2024Counting} . The classical example is the Ehrhart series. For the case where \( A \) contains parameters, we have provided a simple illustration in Example \ref{exam-a}.

\medskip Project 4 aims to explore the application of Equation \eqref{equ-strategies}. Suppose we have obtained a simplicial cone decomposition as in \eqref{equ-dec-K}, but \(\ind(K_1)\) and \(\ind(K_1^*)\) are extraordinarily large. Then we can use Equation \eqref{equ-strategies} to replace \(K_1\) with other cones that might coincide with existing cones.


\end{document}